\documentclass[12pt]{article}

\usepackage{fullpage,amsfonts,amsmath,amsthm,amssymb,mathrsfs,epstopdf,soul}
\usepackage[top=2.54cm, bottom=2.54cm, left=2.54 cm, right=2.54 cm]{geometry}
\usepackage{graphicx}
\usepackage{caption}
\usepackage{subcaption}
 \newcommand{\lab}[1]{\label{#1}}                

 \usepackage[usenames,dvipsnames]{color}
\usepackage[shortlabels]{enumitem}
\newcommand{\jc}[1]{{\bf [Jane:} {\color{blue} #1]}}

\newcommand{\remove}[1]{}
\newcommand\eqn[1]{(\ref{#1})}

\newcommand{\be}{\begin{equation}}
\newcommand{\bel}[1]{\begin{equation}\lab{#1}\ }
\newcommand{\ee}{\end{equation}}
\newcommand{\bea}{\begin{eqnarray}}
\newcommand{\eea}{\end{eqnarray}}
\newcommand{\bean}{\begin{eqnarray*}}
\newcommand{\eean}{\end{eqnarray*}}

\newtheorem{thm}{Theorem}
\newtheorem{cor}[thm]{Corollary}

\newtheorem{lemma}[thm]{Lemma}

\newtheorem{claim}[thm]{Claim}
\newtheorem{prop}[thm]{Proposition}

\newtheorem{remark}[thm]{Remark}

\DeclareMathOperator{\gir}{gir}

\def\proof{\noindent{\em Proof.~}}
\def\qed{~~\vrule height8pt width4pt depth0pt}
\def\ss{{\smallskip}}

\def\model{\texttt{M}}
\def\Nminor{\texttt{N-minor}}
\def\crkc{\texttt{crk=c}}
\def\kCircuit{\texttt{k-circ}}
\def\kConnect{\text{k-conn}}
\def\PGr{\texttt{PGr}}
\def\kColour{\texttt{k-crt}}

\newcommand{\mc}{\mathcal}

\newcommand{\FF}{\mathbb{F}}

\newcommand{\rank}{\text{\textnormal{rk}}}
\newcommand{\corank}{\text{\textnormal{crk}}}
\newcommand{\gbinom}[2]{\genfrac{[}{]}{0pt}{}{#1}{#2}}
\newcommand{\spn}[1]{\langle#1\rangle}

\def\E{{\mathcal E}}

\def\G{{\mathcal G}}

\def\I{{\mathcal I}}

\def\S{{\mathcal S}}

\def\ex{{\mathbb E}}
\def\pr{{\mathbb P}}

\def\bfv{{\boldsymbol v}}

\def\bbF{{\mathbb F}}

\def\eps{\varepsilon}

\date{}

\author{Pu Gao \\
University of Waterloo\\
pu.gao@uwaterloo.ca 
\and
Jacob Mausberg\\
University of Waterloo\\
jmausberg@uwaterloo.ca
\and 
Peter Nelson \\
University of Waterloo\\
apnelson@uwaterloo.ca
}

\title{Evolution of random representable matroids: minors, circuits, connectivity and the critical number}

\begin{document}

\maketitle

\begin{abstract}
We study the evolution of random matroids represented by the sequence of random matrices over $\FF_q$ where columns are  added one after the other, and each column vector is a uniformly random vector in $\FF_q^n$, independent of each other. We study the appearance of matroid minors, the appearance of circuits, the evolution of the connectivities and the critical number. We settle several open problems in the literature.
\end{abstract}

\section{Introduction}

Random graphs were first introduced by Erd\H{o}s and R\'{e}nyi~\cite{erdds1959random} in 1959, and in particular they studied the evolution of random graphs by studying a random process of graphs on a set of $n$ vertices, where the edges are added one after the other in the process. The most astonishing phenomenon in the study of random graph evolution is the characterisation of phase transitions of various (often increasing) graph properties. For instance, initially, the graph is acyclic with each component being a tree of bounded order. Then, small cycles may start to appear whereas each component remains in small size and contains at most one cycle. At the time where the number of edges is around $n/2$, small components start to rapidly connect  to each other and form a giant component in which more complicated graph structures start to appear. Long cycles and graph minors of fixed sizes simultaneously appear at the time the giant component emerges. Other well studied graph properties and graph parameters include connectivity, Hamiltonicity, the appearance of given subgraphs, chromatic number, etc.

The Erd\H{o}s-R\'{e}nyi graph process immediately induces a random process on graphical matroids, which motivates the generalisations to other classes of random matroid processes. One generalisation is to consider the matroids represented by the incidence matrices of random uniform hypergraphs, which was introduced by  Cooper, Frieze and Pegden~\cite{cooper2019minors}, and was more formally described and studied as an evolutionary random process in~\cite{gao2023minors}.
The other generalisation is to consider uniformly random vectors over a finite field and add them one after the other independently, and consider the random matroids represented by these matrices. This model was first introduced by Kelly and Oxley~\cite{kelly1982asymptotic,kelly1982threshold}. In their model, they consider random subsets of the elements in the complete projective geometry $PG(n-1,q)$ where $q$ is a prime power. Two related models were introduced and studied. In the first one, $PG(n-1,q;p)$ where $p\in [0,1]$, every element in  $PG(n-1,q)$ is kept independently with probability $p$. In the second model $PG(n-1,q;m)$ where $m$ is an integer between 0 and $(q^n-1)/(q-1)$, a uniformly random subset of $m$ elements of  $PG(n-1,q)$ is selected. Obviously, $PG(n-1,q;p)$ and $PG(n-1,q;m)$ are analogs of $\G(n,p)$ and $\G(n,m)$ for random graphs, and $PG(n-1,q;m)$ is precisely $PG(n-1,q;p)$ conditioned to $|PG(n-1,q;p)|=m$; i.e. exactly $m$ elements are selected. After that, Kordecki~\cite{kordecki1988strictly,kordecki1996small}, Kordecki and Luczak~\cite{kordecki1991random,kordecki1999connectivity} further studied matroid properties of these models, including circuits, connectivity, and submatroids, etc. Soon after the introduction of $PG(n-1,q;m)$ and $PG(n-1,q;p)$, Kelly and Oxley~\cite{kelly1984random} introduced a slightly different model $M([U_q]_{n\times m})$. In this model,   $[U_q]_{n\times m}$ is a uniformly random $n\times m$ matrix over $\FF_q$, and $M([U_q]_{n\times m})$ is the matroid represented by $[U_q]_{n\times m}$. Due to the independence of the column vectors, $[U_q]_{n\times m}$ is a little easier to analyse than  $PG(n-1,q;p)$ and $PG(n-1,q;m)$. Indeed, as noted by all the authors that followed this study, these three models are asymptotically equivalent for all the problems (e.g.\ rank, circuits and connectivity) they were studying. We give a proof of their equivalence in Proposition~\ref{p:equivalence} below.

In this paper, we use the last model introduced by Kelly and Oxley~\cite{kelly1984random}, however we stress that columns are added one by one and we are interested in the evolution of the matroids represented by this random process of matrices. 
Let $\FF_q$ denote the finite field of order $q$ where $q$ is a prime power. Let $\bfv$ be a uniformly random vector in $\FF_q^n$, and let $(v_i)_{i\ge 1}$ be a sequence of random vectors that are independent copies of $\bfv$. Finally, for every $m\ge 1$, let $A_m=[v_1,\ldots,v_m]$ be the $n\times m$ matrix formed by including the first $m$ vectors $v_1,\ldots, v_m$ in the sequence, and let $M[A_m]$ be the matroid represented by $A_m$. Notice that for every $m\ge 1$, $A_m$ has the same distribution as $[U_q]_{n\times m}$. We study various matroid properties and parameters of $M[A_m]$ as $m$ grows. In particular, we take a thorough study of the time when matroid minors of small ranks appear, the appearance of the projective geometry of growing rank of $n$ as a minor, the appearance of the circuits of different lengths, the evolution of the connectivity, and the growth rate of the critical number. We discuss them in turn in the coming subsections. We start our discussions by unifying the notions of the three models $PG(n-1,q;p)$, $PG(n-1,q;m)$ and $M[A_m]$ and show their asymptotic equivalence. The following Gaussian binomial coefficients will be used throughout the paper, which counts the $k$-dimensional subspaces of $\FF_q^n$: 
$$
\gbinom{n}{k}_q:=\prod_{i=0}^{k-1}\frac{q^{n-i}-1}{q^{k-i}-1}; \quad [n]_q:=\gbinom{n}{1}_q=\frac{q^n-1}{q-1}.  
$$
Notice that $PG(n-1,q)$ contains exactly $[n]_q$ elements.
Given a sequence of probability spaces indexed by $n$, we say a sequence of events $A_n$ occurs asymptotically almost surely (a.a.s.) if $\lim_{n\to\infty}\pr(A_n)=1$. The standard Landau notation and basic matroid notions such as simple matroid, free matroid, and the rank of a matroid will be introduced in Section~\ref{sec:preliminary}.

\begin{prop}\label{p:equivalence} 
    The three models (\model1): $M[A_m]$, (\model2): $PG(n-1,q; m)$ and (\model3): $PG(n-1,q; p)$, are related as follows:
    \begin{enumerate}[(a)] 
        \item (\model1), conditioned on the event that $M[A_m]$ is simple, is equivalent to (\model2).
        \item (\model3), conditioned on the event that precisely $m$ elements of $PG(n-1,q)$ are included, is equivalent to (\model2).
        \item If $m=o(q^{n/2})$ then (\model1) and (\model2) are asymptotically equivalent. 
    \end{enumerate}
\end{prop}

\begin{proof} Parts (a,b) are obvious. For part (c),
    suppose that $m=o(q^{n/2})$.
    By part (i), it suffices to show that a.a.s. $M[A_{m}]$ is simple. The probability that $A_m$ has a zero column is at most $mq^{-n}=o(1)$. Each element in $PG(n-1,q)$ corresponds to exactly $q-1$ vectors in $\FF_q^n$. Thus, the probability that $A_m$ has  two linearly dependent columns is at most $\binom{m}{2} [n]_q (q-1)^2 q^{-2n}=o(1)$. Hence, a.a.s.\ $M[A_m]$ is a simple matroid. \qed
    \end{proof}

\subsection{Minors}
Let $N$ be an $\FF_q$-representable matroid. Let $\tau_{\Nminor}$ be the smallest $m$ such that $M[A_m]$ contains $N$ as a minor (the definition of matroid minor is given in Section~\ref{sec:preliminary}).
Altschuler and Yang~\cite{altschuler2017inclusion} determined the critical window in which $\tau_{\Nminor}$ lies, provided that the size of $N$ is fixed, i.e.\ independent of $n$. Given a matrix or a matroid $M$, let $\corank(M)$ denote the co-rank of $M$.
Given a random variable $X_n$ and a real number $x_n$, we write $X_n=O_p(x_n)$ if
 $$\lim_{\eps\to 0}\limsup_{n\to\infty}\left(\pr(X_n<\eps x_n)+\pr(X_n>\eps^{-1}x_n)\right)=0.
$$

\begin{thm}\label{Altschuler and Yang} 
    Suppose $N$ is a fixed non-free $\FF_q$-representable matroid. 
    \begin{enumerate}[(a)]
    \item (Theorems 5 and 6 of~\cite{altschuler2017inclusion}) Let $k\ge 0$ be a positive integer. There exist constants $C,D>0$ depending on $N$, $q$ and $k$ such that 
    \begin{align*}
    & \liminf_{n\to\infty}\pr(\tau_{\Nminor}\le m) >C \quad \text{if } m=n+k \ \text{and}\ k\ge 1\\
    & \limsup_{n\to\infty}\pr(\tau_{\Nminor}\le m) \le D \quad \text{if } m= n-k.
    \end{align*}
    \item (Theorems 3 and 8 of~\cite{altschuler2017inclusion})  $\tau_{\Nminor} =n+O_p(1)$.
    \end{enumerate}
\end{thm}
\vspace{0.3cm}

Note that if $N$ is a free matroid with rank $r$ such that $n-r\to\infty$ then it is easy to see (e.g.\ it follows easily from the proof of Lemma~\ref{full row rank}) that a.a.s.\ all the columns of $A_r$ are linearly independent and thus a.a.s.\ $\tau_{\Nminor}=r$. On the other hand, if $N$ is  not $\FF_q$-representable then $N$ can never be a minor of $M[A_m]$, no matter how large $m$ is. Therefore, in the discussions of matroid minors we only focus on non-free $\FF_q$-representable matroid.

Altschuler and Yang gave specific bounds $C,D$ in Theorem~\ref{Altschuler and Yang}(a). Interested readers can find them in~\cite{altschuler2017inclusion}. We do not express them here, as these bounds  only give qualitative information about $\pr(|\tau_{\Nminor}-n|>k)$ for $k\to\infty$, which they used to come to the conclusion of Theorem~\ref{Altschuler and Yang}(b). However these bounds give little information about $\pr(\tau_{\Nminor}\le n+k)$ when $|k|$ is small. Our first result is a strengthening of Theorem~\ref{Altschuler and Yang}(a) by providing the precise limiting distribution of $\tau_{\Nminor}-n$. For convenience, we define $\sum_{i=j}^h a_i$ to be 0 and $\prod_{i=j}^h a_i$ to be 1 if $h<j$ for any sequence of real numbers or real functions $a_i$. Given a power series $P(z)$ let $[z^n] P(z)$ denote the coefficient of $z^n$ in $P(z)$.

\begin{thm}\label{thm:point-prob}     Let $N$ be a fixed $\FF_q$-representable matroid and let $r$ and $c$ denote the rank and the co-rank of $N$ respectively. Suppose that $c\ge 1$.
 Then, for any fixed $k\in {\mathbb Z}$,
 \[
\lim_{n\to\infty} \pr(\tau_{\Nminor}=n+k)=C_{c,k},
 \]
where $C_{c,k}=0$ if $k>c$; and if $k\le c$ then
\begin{eqnarray*}
C_{c,k}&=&\beta_{c,k} q^{k-c} \sum_{i=0}^{c-1}\left(\alpha_{c,k,0}+ \sum_{i=1}^{c-1} \frac{1}{\prod_{j=1}^i(1-q^{-j})} \alpha_{c,k,i}\right),\\
\beta_{c,k}&=&\prod_{j=c+1-k}^{\infty}(1-q^{-j}),\\
\alpha_{c,k,i}&=&[z^{c-1-i}]\prod_{j=0}^{c-k-1}(1-zq^{-j})=(-1)^{c-1-i}\sum_{*} q^{-\sum_{r=1}^{c-1-i} j_r},
\end{eqnarray*}
where the summation $\sum_{*}$ in the expression for $\alpha_{c,k,i}$ is over all integers $0\le j_1<j_2<\ldots<j_{c-1-i}\le c-k-1$.
\end{thm}

We plot below the limiting point-wise probabilities $\pr(\tau_{\Nminor}=n+k)$ for $k\in [-10,5]$ when $(q,c)=(2,1)$ (the one on the left) and when $(q,c)=(2,2)$ (the one in the middle). The last figure on the right compares the limiting cumulative distribution function $\pr(\tau_{\Nminor} \le n+k)$, in red dots, with the bounds in Theorem~\ref{Altschuler and Yang}(a)  by Altschuler and Yang, in green curve. Recall that the bounds in Theorem~\ref{Altschuler and Yang}(a) are upper bounds for non-positive $k$ and lower bounds for positive $k$.

\noindent\begin{minipage}{0.34\textwidth}
    \includegraphics[width=\linewidth]{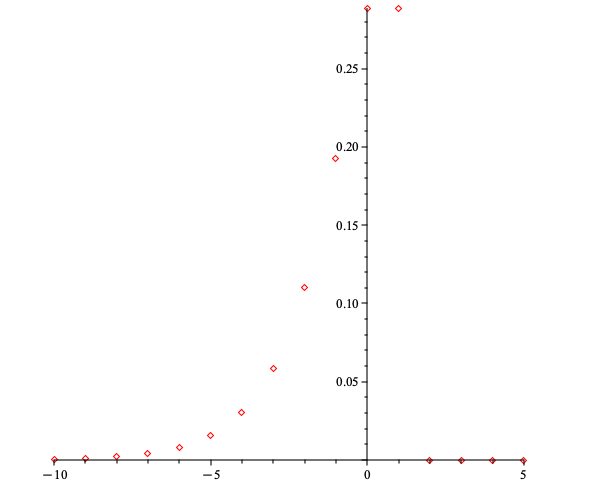}
\end{minipage}%
\hfill%
\begin{minipage}{0.35\textwidth} 
\includegraphics[width=\linewidth]{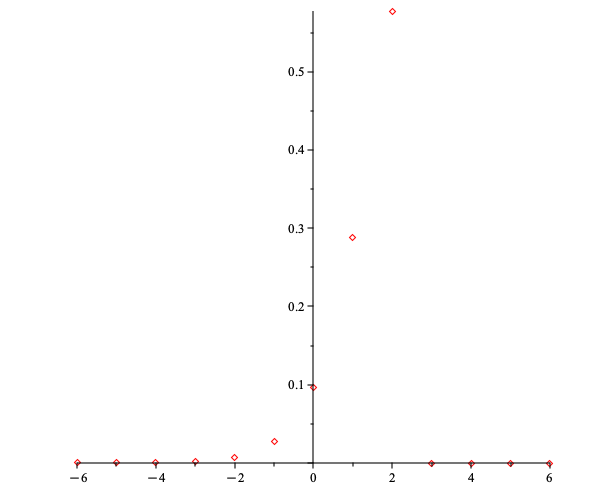}
\end{minipage}
\hfill%
\begin{minipage}{0.30\textwidth} 
\includegraphics[width=\linewidth]{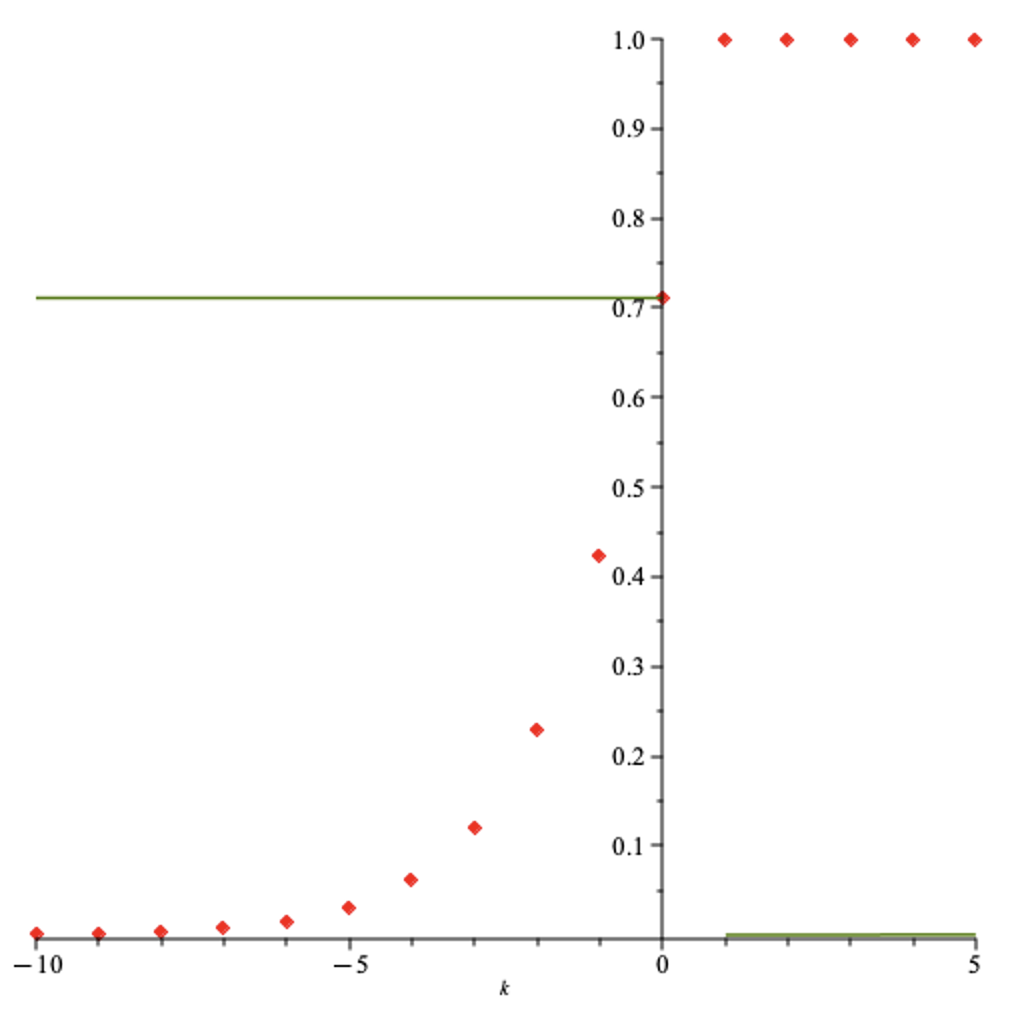}
\end{minipage}

Our second result improving over Altschuler and Yang's is a hitting time result of $\tau_{\Nminor}$ when the rank and the co-rank of $N$ are fixed or slowly growing functions of $n$. Let $\tau_{\crkc}$ be the smallest $m$ such that the co-rank of $A_m$ is equal to $c$. This is well defined as the co-rank of $A_m$ is non-decreasing as $m$ grows.

\begin{thm}\label{very small minor}
    Suppose that $N$ is an $\FF_q$-representable non-free matroid such that $rq^{rc}=o(n)$, where $r$ and $c$ denote the rank and the co-rank of $N$ respectively.
    Then a.a.s. $\tau_{\Nminor}=\tau_{\crkc}$.
\end{thm}


As a direct corollary, we prove that a.a.s.\ an $N$-minor is formed by step $n+\corank(N)$, improving Theorem~\ref{Altschuler and Yang}(b). On the other hand, there is a non-vanishing probability that the first $N$-minor is created precisely during the step $n+\corank(N)$.  
\begin{cor}\label{thm:one_step}
Suppose that $N$ is a fixed non-free $\FF_q$-representable matroid. Let $c$ be the co-rank of $N$. Then, a.a.s.\ $\tau_{\Nminor}=n+O_p(1)$. Moreover, letting $\gamma_{q,c}=\prod_{j=c}^{\infty} (1-q^{-j})$, we have
\[
\lim_{n\to\infty}\pr(\tau_{\Nminor}\le n+c)=1;\quad \lim_{n\to\infty} \pr(\tau_{\Nminor}\le n+c-1)= 1-\gamma_{q,c},
\] 
for every $(q,c)$. Moreover,
\begin{align*}
\lim_{q\to \infty} \gamma_{q,c} =1 \quad \mbox{for every $c$};\quad
\lim_{c\to \infty} \gamma_{q,c} =1 \quad\mbox{for every $q$}.
\end{align*}
\end{cor}
The next corollary shows that $\tau_{\Nminor}$ has a 1-point concentration if $|N|$ is not too large, and $\corank(N)=\omega(1)$. Here, $|N|$ denotes the size of $N$, which is the number of elements in $N$.
\begin{cor}\label{cor:minor} Suppose that $N$ is an $\FF_q$-representable non-free matroid. If $\corank(N)=\omega(1)$ and $|N|\le (2-\eps)\sqrt{\log_q n}$ for some fixed $\varepsilon>0$, then a.a.s.
$\tau_{\Nminor}=n+\corank(N)$.
\end{cor}
Note that the matroid $N$ in Theorems~\ref{Altschuler and Yang},~\ref{thm:point-prob},~\ref{very small minor} and Corollaries~\ref{thm:one_step} and~\ref{cor:minor} is general, and does not need to be simple. In the next corollary of Theorem~\ref{very small minor}, we generalise Theorem~\ref{Altschuler and Yang} and study $\tau_{\PGr}$, the minimum integer $m$ such that $M[A_m]$ contains the complete projective geometry $PG(r-1,q)$ as a minor. All the logarithms in the paper are natural logarithms with base $e$, unless otherwise with a specified base.

\begin{thm} \label{thm:PGr} Suppose that $r=\omega(1)$. Let $\zeta=[r]_q$.
\begin{enumerate}[(a)] 
\item A.a.s.\ $\tau_{\PGr}\le n+\zeta\log \zeta +\omega(\zeta) $.
\item If $r=\omega(\log n)$ then a.a.s.\ $\tau_{\PGr}\sim \zeta\log \zeta$.
\item If $N$ is an $\FF_q$-representable simple non-free matroid with rank $r\le \alpha \log n$ for any fixed $\alpha<1/\log q$ then a.a.s.\ $\tau_{\Nminor}\sim n$.
\end{enumerate}
\end{thm}

\begin{remark}
Theorem~\ref{thm:PGr} determines the asymptotic value of $\tau_{PGr}$ provided that $r\le (1-\eps)\log_q n$ for some fixed $\eps>0$, or $r=\omega(\log n)$. However, we did not manage to prove a lower bound for $\tau_{PGr}$ that matches its upper bound in Theorem~\ref{thm:PGr}(a) for $\log_q n\le r=\Theta(\log n)$. There is a trivial lower bound $n+\zeta-r=n+(1+o(1))\zeta$, since the co-rank of $A_m$ must be at least the co-rank of $PG(r-1,q)$ (see Lemma~\ref{minor corank} below). It is possible that in the range $\log_q n\le r=\Theta(\log n)$, both the upper bound (Theorem~\ref{thm:PGr}(a)) and the lower bound ($n+(1+o(1))\zeta$) are not tight.
\end{remark}

The proofs of Theorems~\ref{thm:point-prob},~\ref{very small minor} and~\ref{thm:PGr}, and the proofs of Corollaries~\ref{thm:one_step} and~\ref{cor:minor} will be presented in Section~\ref{sec:minor}.

\subsection{Circuits}

The circuits in a matroid (see its formal definition in Section~\ref{sec:preliminary}) are minimal dependent subsets and are analogs of cycles in a graph.   The appearance of the circuits with constant length has been studied by Kelly and Oxley~\cite{kelly1984random}, whereas the longer circuits were investigated by Kordecki and {\L}uczak~\cite{kordecki1999connectivity}. We collect and restate their results in the following theorem. 

\begin{thm}\label{thm:old}
\begin{enumerate}[(a)]
\item (Theorem 5.1 of~\cite{kelly1984random})
    Suppose that $k=o(m)$ and that $m^kq^{-n}$ is bounded.
    Then $$\Pr(M[A_m]\text{has no $k$-circuits})\sim\exp\left(-\frac{(q-1)^{k-1}}{k!}\frac{m^k}{q^n}\right).$$
    \item (Theorem 6 of~\cite{kordecki1999connectivity})     For each $1\le k\le n+1$, let
    $\mu_k=\binom{m}{k}(q-1)^kq^{-n}.$
    \begin{enumerate}[(i)]
        \item If $\mu_k\to0$, then a.a.s.\ $M[A_{m}]$ does not contain a $k$-circuit.
        \item If $\mu_k\to\infty$ and either $n-k\to\infty$ or $k(m-k)/m\to\infty$, then a.a.s.\ $M[A_{m}]$ contains a $k$-circuit. 
    \end{enumerate}
    \end{enumerate}
\end{thm}

\begin{remark} Kordecki and {\L}uczak's result~\cite{kordecki1999connectivity} was proved for $PG(n-1,q;m)$. The same holds for $M[A_m]$ by Proposition~\ref{p:equivalence}. The original statement of~\cite[Theorem 6]{kordecki1999connectivity}
was slightly stronger, by considering the appearance of circuits whose lengths lie in a specified interval. For simplicity we stated a simplified version. The value $\mu_k$ is the asymptotic number of circuits with length $k$ in $PG(n-1,q;m)$ and $M[A_m]$. 
\end{remark}

Theorem~\ref{thm:old} provides the full information on the types of circuits that are likely or unlikely to appear in $M[A_m]$ for every $m$. We give a few interesting corollaries of Theorem~\ref{thm:old} in terms of $\tau_{\kCircuit}$, the minimum integer $m$ such that $M[A_m]$ has a circuit with length $k$.  Given $0<a\le 1$, define 
\begin{equation}
    g_a(y)=y\log y+a\log(q-1)-a\log a-(y-a)\log(y-a)-\log q,\quad \text{for}\ y\ge a, \label{eq:g}
    \end{equation}
    where $0\log 0$ is defined to be 0 so that $g_a(y)$ is continuous at $y=a$. 
\begin{cor} \label{cor:circuits}
\begin{enumerate}[(a)]
\item Let $k$ be a fixed positive integer. Then, $\tau_{\kCircuit}=\Theta_p(q^{n/k})$. 
\item If $k\to\infty$ and $k=o(n)$ then a.a.s.\ $\tau_{\kCircuit}\sim\frac{1}{e(q-1)}kq^{n/k}$.
\item If $k\sim an$ for some fixed $0<a\le 1$ then $g_a(y)$ has a unique root $b$, and a.a.s.\ $\tau_{\kCircuit}\sim bn$.
\end{enumerate}
\end{cor}

\begin{remark}\label{remark:g}
    For $0<a\le 1$, let $b=b(a)$ be the unique root of $g_a(y)$. 
    Then $b$ is strictly convex and has a minimum value of $1$, achieved at $a^*:=\frac{q-1}{q}$.
    Moreover, $b(1)<2$ and $b(a)\to\infty$ as $a\to 0$. 
\end{remark}

\begin{figure}[h]
\centering
\begin{minipage}{.5\textwidth}
  \centering
  \includegraphics[width=.6\linewidth]{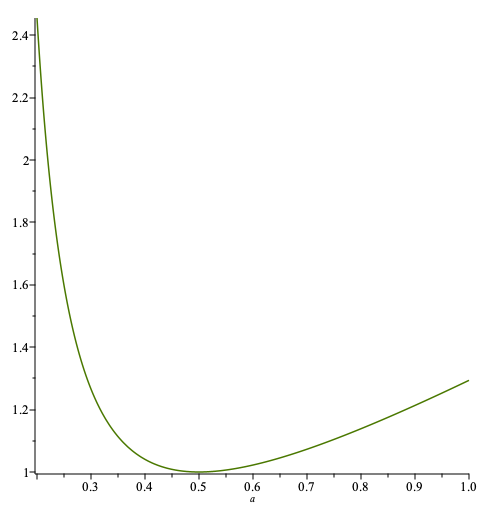}
  \captionof{figure}{Plot of $b(a)$}
  \label{fig:b}
\end{minipage}%
\begin{minipage}{.5\textwidth}
  \centering
  \includegraphics[width=.65\linewidth]{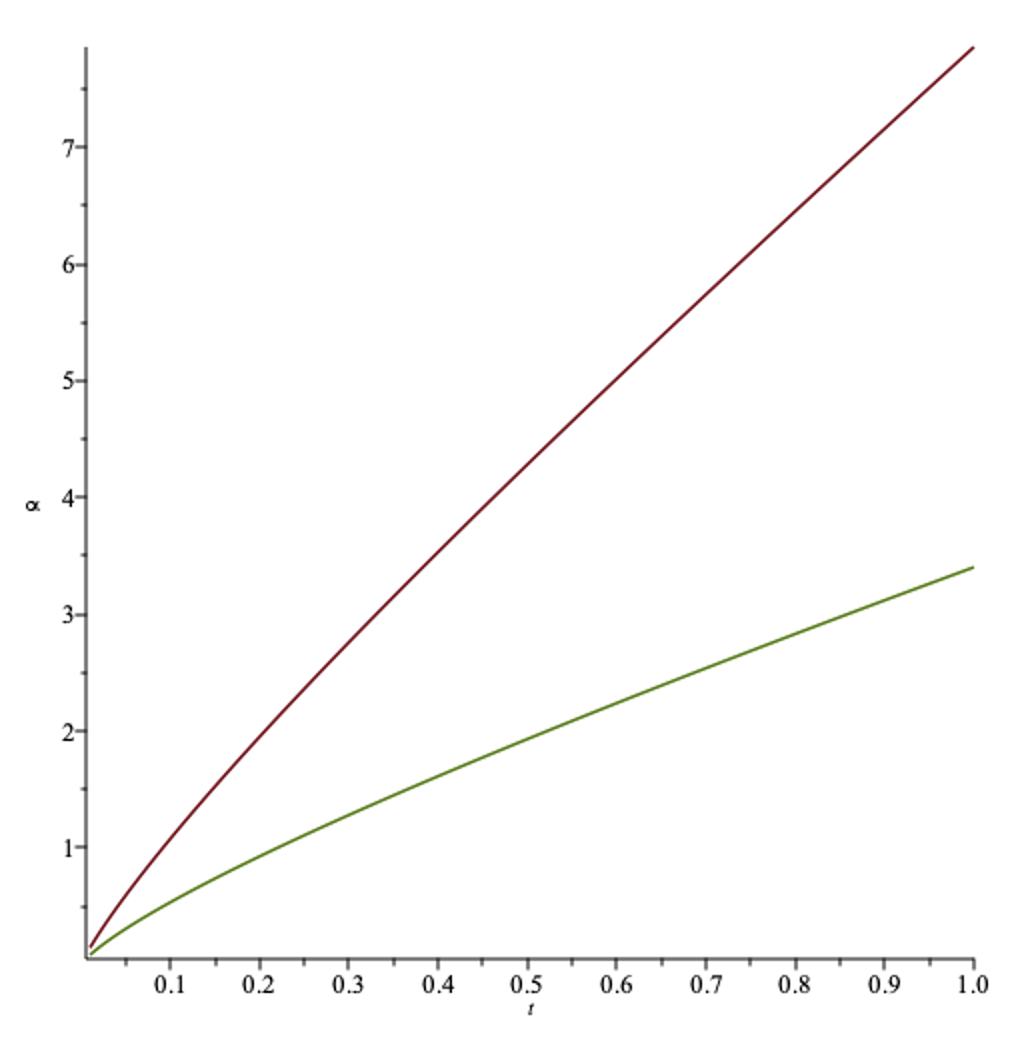}
  \captionof{figure}{Upper bound of $\alpha$ from Kelly and Oxley and our lower bound}
  \label{fig:bounds}
\end{minipage}
\end{figure}

Remark~\ref{remark:g} follows from simple calculus and we include its proof in the Appendix for interested readers. A plot of $b(a)$ for $q=2$ is given below in Figure~\ref{fig:b}.  Following Remark~\ref{remark:g} we immediately obtain the following interesting corollary about the length of the first appearing circuit and the appearing time of the  first Hamilton circuit (i.e.\ the circuit with length $n$).

\begin{cor}\label{cor2:circuits}
\begin{enumerate}[(a)]
\item The first circuit appearing in $M[A_m]$ has length asymptotic to $(1-q^{-1}) n$;
\item The first Hamilton circuit  appears in less than $2n$ steps for every finite field $\FF_q$.
 \end{enumerate} 
\end{cor}
The proofs for Corollaries~\ref{cor:circuits} and~\ref{cor2:circuits} are  given in Section~\ref{sec:circuits}.

\remove{
\begin{figure}[h]
\begin{center}
\includegraphics[scale=0.35]{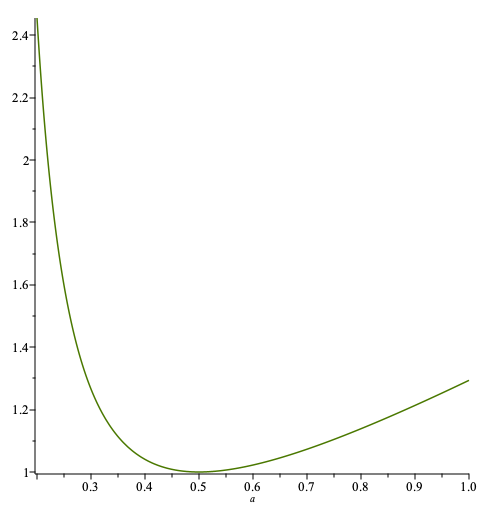}
\caption{Plot of $b(a)$ for $a\in (0,1)$ for $q=2$}
\label{fig:distribution}
\end{center}
\end{figure}
}
\subsection{Connectivity}

The concept of vertex-connectivity in graphs generalizes naturally to matroids.
It is easily seen that a graph $G$  is $k$-vertex-connected if and only if
it is not the union of two edge-disjoint subgraphs $G_1,G_2$ such that $\min\{|V(G_1)|, |V(G_2)|\} > k$,  and $|V(G_1) \cap V(G_2)| < k$;
the latter is equivalent to $|V(G_1)| + |V(G_2)| - |V(G)| < k$.
 
Analogously, a \emph{vertical $k$-separation} in a matroid $M$ is a partition $(A_1,A_2)$ of $M$ so that 
\begin{enumerate}[(i)]
\item\label{nontriv} $\min\{\rank(A_1),\rank(A_2)\} \ge k$, and 
\item\label{lambda} $\rank(A_1)+\rank(A_2)-\rank(M) \le k-1$. 
\end{enumerate}
The \emph{vertical connectivity} $\kappa(M)$ is defined to be the smallest $k$ such that $M$ has a vertical $k$-separation. 
It may be the case that $M$ has no vertical $k$-separations for any $k$, in which case $\kappa(M) = \infty$\footnote{some work defined $\kappa(M)$ to be the rank of $M$ in this case}; 
this holds, for instance, if $M = \bbF_q^n$.

From the perspective of graph theory, this is the most natural notion of matroid connectivity; 
indeed, if $G$ is a graph with no isolated vertex, then the value of $\kappa$ for the graphic matroid $M(G)$ 
is equal to the vertex connectivity of $G$: see \cite{oxley2011matroid}, Theorem 8.6.1. 
(Incidentally, this fact is the reason for the unnatural-looking offset by one in condition \ref{lambda}). 

However, $\kappa$ fails to have a certain natural matroid property: invariance under duality. 
Each matroid $M$ has a dual matroid $M^*$, and the relationship between $M$ and $M^*$ is crucial 
in much of matroid theory -- for instance, we have $M^{**} = M$, and matroid duality agrees with planar duality
in the case of the graphic matroids of planar graphs.
It is not necessary here to discuss matroid duality in detail, but we comment that 
the rank function of the dual matroid is given by $\rank^*(X) = |X| + \rank(E(M) \setminus X) - \rank(M)$
(see \cite{oxley2011matroid}, Proposition 2.1.9). 

In general, we have $\kappa(M) \ne \kappa(M^*)$, so vertical connectivity has a dual notion. 
A \emph{cyclic $k$-separation} of $M$ is a vertical $k$-separation of $M^*$,
and the \emph{cyclic connectivity} $\kappa^*(M)$ is the smallest $k$ such that $M$ has a cyclic $k$-separation. Similarly as before, $\kappa^*(M)$ defined to be $\infty$ if no cyclic separation exists.
Using the above formula for the dual rank function, one can easily show that a partition $(A_1,A_2)$ of $M$ is a cyclic $k$-separation of $M$ if and only if 
\begin{enumerate}[(i)]
\item\label{nontrivcyc} $\rank(A_i) < |A_i|$ for each $i$, and 
\item\label{lambdacyc}$\rank(A_1)+\rank(A_2)-\rank(M) \le k-1$. 
\end{enumerate}

Condition \ref{nontrivcyc} can be replaced with the requirement that $A_1$ and $A_2$ are dependent, or 
by the condition $\min(\rank^*(A_1), \rank^*(A_2)) \ge k$. 
It is harder to relate this intuitively to graphs, except to comment that, if $M$ is the graphic matroid 
of a planar graph $G$, then the cyclic connectivity of $M$ is the vertex connectivity of a planar dual of $G$.
One way to construct a small cyclic separation in a matroid $M$ is simply to take a small circuit; 
if $C$ is a circuit of $M$ for which $E(M) \setminus C$ is a dependent set, then 
$(C, E(M) \setminus C)$ satisfies \ref{nontrivcyc}, and $\rank(C) + (\rank(E(M) \setminus C) - \rank(M)) \le |C|-1 + 0$, 
so $C$ gives a cyclic $|C|$-separation, implying that $\kappa^*(M) \le |C|$. 
In the setting of planar duality, this corresponds to a fact that a small cycle in the planar dual of $G$ 
gives rise to a small cut in $G$. 

Evidently cyclic connectivity is not invariant under matroid duality. However, there is a third notion of 
connectivity that is. A \emph{Tutte $k$-separation of $M$} is a partition $(A_1,A_2)$ of $M$ so that  
\begin{enumerate}[(i)]
\item $\min\{|A_1|,|A_2|\}\ge k$, and 
\item $\rank(A_1)+\rank(A_2)-\rank(M) < k-1$. 
\end{enumerate}
The \emph{Tutte connectivity} $t(M)$ is the smallest $k$ so that $G$ is Tutte $k$-connected. 
Each cyclic or vertical $k$-separation is also a Tutte $k$-separation, which implies that $t(M) \le \min(\kappa(M), \kappa^*(M))$. 
In fact, one can show that if $|M| \ge 3$ and $t(M) < \infty$, 
then there is always either a vertical or a cyclic $t(M)$-separation, 
which implies that equality holds.
This fact is essentially stated in \cite{oxley2011matroid}, Proposition 8.6.6, but the notation is slightly 
different from ours in edge cases, as $\kappa(M)$ and $\kappa^*(M)$ are defined in \cite{oxley2011matroid} to always be finite. 
\begin{prop}
    If $M$ is a matroid with $|M| \ge 3$, then $t(M) = \min\{\kappa(M), \kappa^*(M)\}$.    
\end{prop}
This fact shows that Tutte connectivity is invariant under matroid duality, so in a sense it is 
the most natural connectivity notion of the three; usually the unadorned `connectivity' of a matroid
refers to the Tutte connectivity. 

The \emph{girth} $\gir(M)$ of a matroid $M$ is the size of a smallest circuit of $M$, or $\infty$ if $M$ is independent. 
We have seen that in nontrivial cases, small circuits give small cyclic separations. It follows that the girth usually
provides an upper bound for the connectivity of a matroid. The following bound (\cite{oxley2011matroid}, Theorem 8.6.4) 
will be useful.
\begin{prop}\label{p:tutte}
    If $M$ is a matroid that is not a uniform matroid $U_{r,n}$ with $n \ge 2r-1$, then $t(M) = \min(\kappa(M), \gir(M))$. 
\end{prop}

A striking difference between matroid and graph connectivities is the monotonicity. From the definitions it is easy to see that all of the Tutte, the vertical, and the cyclic connectivities are not monotone; i.e.\ adding elements to a matroid may decrease the connectivity. Our first observation is that $\kappa(M[A_m])$ is a.a.s.\ monotonely non-decreasing as $m$ grows. On the other hand, the Tutte connectivity first follows $\kappa(M[A_m])$, and then after some linear number of steps, it becomes governed by $\text{gir}(M[A_m])$ and decreases as $m$ grows. The evolutionay trajectory of $\kappa^*(M[A_m])$ is very different from $\kappa(M[A_m])$, which starts from $\infty$ and in the end is governed by $\text{gir}(M[A_m])$. We will study $\kappa^*(M[A_m])$ in a different paper.

\begin{thm}\label{thm:monotone}
\begin{enumerate}[(a)]
\item A.a.s.\ $\kappa(M[A_m])$ is monotonely non-decreasing as $m$ increases.
\item A.a.s.\ $t(M[A_m])=\kappa(M[A_m])$ for all $m=n+o(n)$.
\item A.a.s.\ there exists $\hat m=\Theta(n)$ such that $t(M[A_m])=\kappa(M[A_m])$ for all $m<\hat m$, and  $t(M[A_m])<\kappa(M[A_m])$ for all $m\ge \hat m$.

\end{enumerate}
\end{thm}

Thanks to the monotonicity of the vertical connectivity of $M[A_m]$ as shown in Theorem~\ref{thm:monotone}(a), it is natural to define $\tau_{\kConnect}$ to be the smallest $m$ such that $M[A_m]$ is vertically $k$-connected. Due to Proposition~\ref{p:tutte} and our good understanding of $\tau_{\kCircuit}$ from Corollary~\ref{cor:circuits},  $\tau_{\kConnect}$ and $\tau_{\kCircuit}$ together would immediately determine the evolution of $t(M[A_m])$.  The limiting distribution of $\tau_{\kConnect}$ has been determined by Kordecki and {\L}uczak~\cite{kordecki1999connectivity} when $k$ is a constant; whereas an upper bound on $\tau_{\kConnect}$ was provided by Kelly and Oxley~\cite{kelly1984random} when $k$ is linear in $n$.

\begin{thm}\cite[Theorem 4]{kordecki1999connectivity}
    Let $k\ge 2$ and $m-n-(k-1)\log_qn\to c$ for some constant $c$.
    Then $$\lim_{n\to\infty}\pr\big(M[A_{m}]\text{ is $k$-connected}\big)=\exp\big(-(q-1)^{k-2}q^{-c}/(k-1)!\big).$$
\end{thm}

Kelly and Oxley give the following upper bounds for when $M[A_{n\times m}]$ becomes $k$-connected.
\vspace{0.3cm}

\begin{thm} \label{thm:old_connect}
\begin{enumerate}[(a)]
\item (\cite[Theorem 4.4]{kelly1984random}) A.a.s.\ $\kappa(M[A_m])=\infty$\footnote{In the original work of~\cite{kelly1984random}, the theorem was phrased as ``$M[A_m]$ is vertically $n$-connected''; we rephrased it to be consistent with our definition of vertical connectivity when no separator exists.} if $m\ge (1+\alpha)n$ where $\alpha$ is any constant such that 
$$\alpha>\frac{\log(2q-1)}{2\log q-\log(2q-1)}.$$
    \item (\cite[Theorem 4.5]{kelly1984random} ) Suppose that $k\sim tn$ for some fixed $0<t<1$ then a.a.s.\ $\tau_{\kConnect}\le (1+\alpha)n$  for any constant $\alpha$ such that 
$$t\log\left[\frac{(1+t)\alpha}{t^2}\right]<(\alpha-t)\ln q-2t.$$
    \end{enumerate}
\end{thm}

Our contributions are the determination of the sharp phase transition of $\tau_{\kConnect}$ for all $k=o(n)$, and a lower bound on $\tau_{\kConnect}$ for $k=\Theta(n)$. 
\begin{thm} \label{sublinear connectivity}
\begin{enumerate}[(a)]
\item  Suppose $k\to\infty$ and $k=o(n)$. Then, a.a.s.\ $\tau_{\kConnect}-n\sim k\log_q (n/k)$.
\item  Suppose that $k\sim tn$ for some fixed $0<t< 1$. Then, a.a.s.\ $\tau_{\kConnect}\ge (1+\alpha)n$ where $\alpha$ is any constant satisfying the following.
\begin{equation}
t\log\frac{1+\alpha}{t}+(1+\alpha-t)\log\frac{1+\alpha}{1+\alpha-t}+t\log(q-1)-\alpha\log q >0. \label{eq:LB_connect}
\end{equation}
\end{enumerate}
\end{thm}
Our lower bound on $\tau_{\kConnect}$ for linear $k$ does not match the upper bound in Theorem~\ref{thm:old_connect}. See Figure~\ref{fig:bounds} for the plot in which we compare the upper bounds by Kelly and Oxley and the lower bounds in Theorem~\ref{sublinear connectivity}(b); the horizontal axis is for $t=k/n$, and the vertical axis is for $n^{-1}\tau_{\kConnect}-1$. Determining the asymptotic value of $\tau_{\kConnect}$ for $k=\Theta(n)$ is an interesting open problem. 

The proofs for Theorems~\ref{thm:monotone} and~\ref{sublinear connectivity} will be given in Section~\ref{sec:cconectivity}.

\remove{
\begin{figure}[h]
\begin{center}
\includegraphics[scale=0.3]{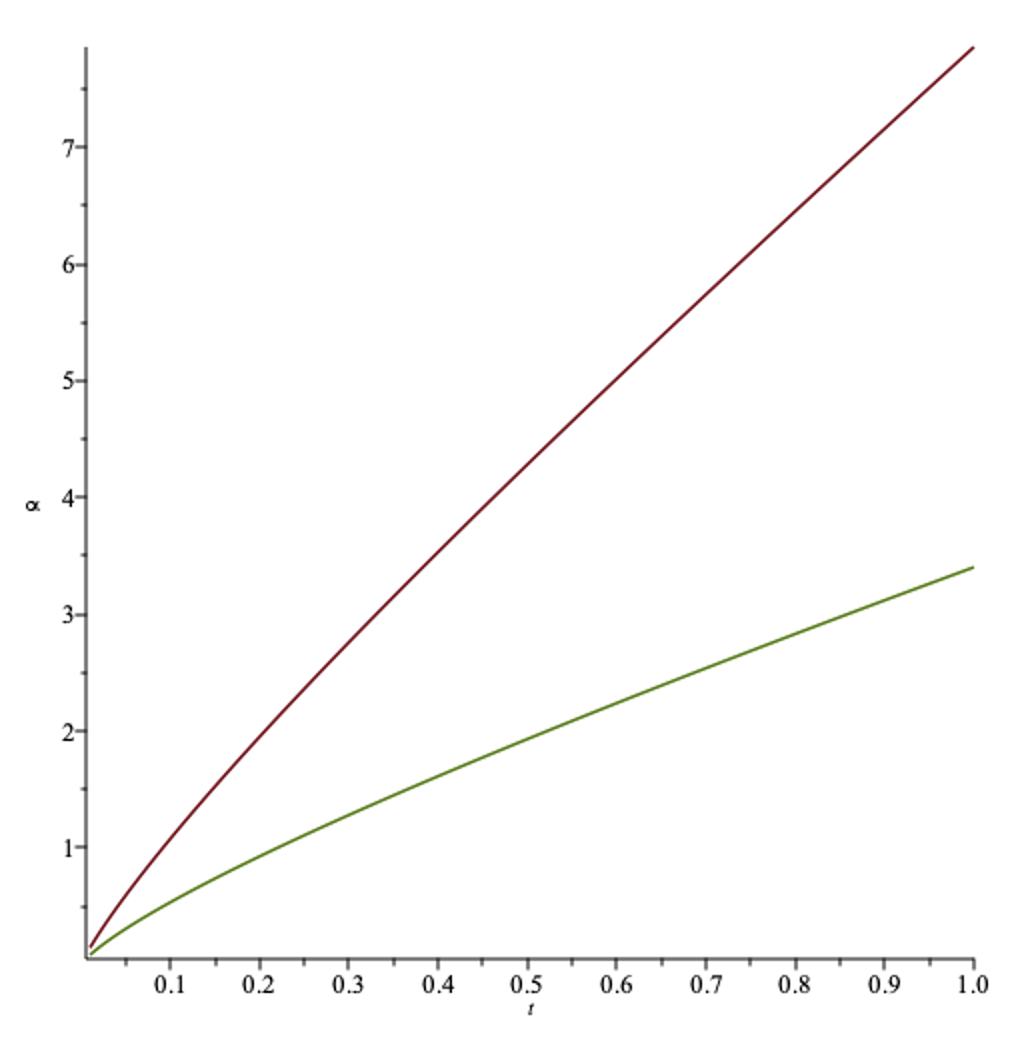}
\caption{Upper bound of $\alpha$ from Kelly and Oxley and our lower bound.}
\label{fig:bounds}
\end{center}
\end{figure}
}

\subsection{Critical number}

The critical number is an extension of the notion of the chromatic number of graphs to matroids. 
It is easy to see that a graph $G$ is $2^k$-colourable if and only if $E(G)$ is the union of $k$ edge cuts of $G$. 
For each set $U \subseteq V(G)$, the edge cut $\delta(U)$ corresponds to the set of support-$2$ vectors in $\bbF_2^V$ that have 
a nonzero dot product with the characteristic vector $1_U$. Hence, we can form an analogue of chromatic number by 
defining $\chi_q(M)$ for each matroid $M \subseteq \bbF_q^n$ to be the minimum $k$ such that there are vectors 
$v_1, \dotsc, v_k \in \bbF_q^n$ such that, for all $w \in M$, the dot product $w^Tv_i$ is nonzero for some $i \in [k]$. In other words, $\chi_q(M)$ is the minimum integer $k$
such that there exists an $(n-k)$-dimensional subspace $S$ of $\FF_q^n$ such that $M\cap S=\emptyset$.
We call $\chi_q(M)$ the \emph{critical number} of $M$; it has previously been called the critical exponent. 
It can be seen as the right analogue of chromatic number in a variety of contexts; see \cite{oxley2011matroid}, p. 588 
for a discussion.


Our goal is to determine when $M[A_{m}]$ has critical number $k$ for all $1\le k\le n$.
(In the range of values for $m$ that we consider, $A_{m}$ will a.a.s.\ not contain a zero column, so the critical number is well-defined).
Notice that the critical number starts at 1 (i.e. when $m=1$) and increases monotonically with $m$ (i.e. as more columns are added).
Also, notice that the critical number cannot skip over any values of $k$, since adding one column can increase the critical number by at most one.
Thus, we focus on determining the step when the critical number of $M[A_{m}]$ jumps from  $k$ to $k+1$ for $1\le k\le n-1$.

Let $\tau_{\kColour}$ be the minimum integer $m$ such that $\chi_q(M[A_m])=k+1$. In other words, $\tau_{\kColour}$ is the precise step where the critical number jumps from $k$ to $k+1$. We obtain the following theorem whose proof will be presented in Section~\ref{sec:colouring}.
\begin{thm}\label{thm:colouring}
Let $k$ be a positive integer such that $k\le n-\log_q n-\log_q\log n-\omega(1)$. Then, a.a.s.\ $\tau_{\kColour}\sim -k(n-k)\log q/\log(1-q^{-k})$.
\end{thm}

\begin{remark}
Note that our theorem covers all positive integers $k$ up to distance $\log_q n+\log_q\log n+\omega(1)$ from $n$. For greater $k$ up to $n-1$ we have the trivial asymptotic upper bound $nq^n\log q$ on $\tau_{\kColour}$ from coupon collection (with $q^n$ coupons). 
\end{remark}

\subsection{Other related work}

Other than the rank, circuits, connectivity, critical number, minors that are discussed in this paper, the thresholds and limiting distributions of the number of small submatroids were studied by Oxley~\cite{oxley1984threshold}, and further extended by Kordecki~\cite{kordecki1988strictly,kordecki1996small}.

The uniformly random matroid on $n$ elements was introduced and studied by Mayhew, Newman, Welsh and Whittle~\cite{mayhew2011asymptotic}. Research in this direction focuses on enumeration of matroids and matroid extensions. We refer the readers to~\cite{lowrance2013properties,bansal2015number,nelson2016almost,knuth1974asymptotic,pendavingh2018number} for results in this field.

\section{Preliminary}
\lab{sec:preliminary}

A matroid is defined on a pair $M=(E,\I)$ where $E$ is a set called the ground set of the matroid $M$, and $\I\subseteq 2^E$ denotes the set of independent sets of $M$. The size of $M$ is $|E|$, the number of elements in the ground set. The rank of $M$,  denoted by $\rank(M)$, is the size of a largest independent set. The co-rank of $M$, denoted by $\corank(M)$, is defined by $|E|-\rank(M)$.  Let $\FF$ be a field. A matroid $M=(E,\I)$ is said $\FF$-representable if there is a matrix $A=[a_p]_{p\in E}$ over $\FF$, where the columns of $A$ are indexed by elements in $E$, such that $S\subseteq E$ is an independent set if and only if $\{a_p: p\in S\}$ is a linearly independent set. A matroid $M$ is free if $E(M)$ is an independent set, and $M$ is simple if $M$ does not have dependent subsets of cardinality one or two. Consequently, if $M$ is represented by a matrix $A$ over $\FF$, then $M$ is free if all columns of $A$ are linearly independent, and $M$ is simple if $A$  does not contain the zero column, or two linearly dependent columns. The uniform matroid $U_{r,n}=([n],\I)$ is the matroid on ground set $[n]$ such that $\I$ consists of all subsets of $[n]$ of cardinality at most $r$. A circuit of a matroid is a minimal dependent subset of elements in $M$. In other words, every proper subset of a circuit is an independent set. The length of a circuit is the number of elements in the circuit. 

Suppose that $M=(E,\I)$ and $X\subseteq E$. The deletion of $X$ from $M$ is defined by $M\setminus X=(E\setminus X, \{I\in\I: I\cap X=\emptyset\})$. The contraction of $X$ from $M$ is defined by $M/X=(E\setminus X, \I^*)$ where
$$
\I^*=\{I\in\I: I\cap X=\emptyset, I\cup J\in \I\ \text{for some maximal independent subset $J$ of $X$}\}.
$$
A submatroid of $M$ is any matroid obtained by deleting a subset of elements in $M$; whereas a minor of $M$ is a matroid obtained by deleting and contracting elements in $M$. If $M$ is a matroid represented by a matrix $A$ over a field $\FF$, then there are matrix operations described as follows which yield representations for submatroids and minors of $A$.

Given a matrix $A$ over $\FF$ where columns are indexed by $E$, and $X\subseteq E$, let $A_X$ denote the matrix obtained from $A$ by only including columns in $X$.  Let $A_{E/X}$ be any matrix obtained by first obtaining matrix $B\sim A$ via row operations such that 
$$
B_{X}=\left[
\begin{array}{cc}
I_{a\times a} & *\\
0 & 0
\end{array}
\right],   \quad \text{where $a=\rank(A_X)$},
$$
and then deleting the $a$ rows where $I_{a\times a}$ lies, together with all the columns in $X$.
Suppose that $A=[a_p]_{p\in E}$ is a matrix representing $M=(E,\I)$, and $X\subseteq E$. Then, $M\setminus X$ is represented by $A_{E\setminus X}$ and $M/X$ is represented by $A_{E/X}$. We refer the readers to Oxley~\cite{oxley2011matroid} for other basics in matroid theory. 

Finally, we use standard Landau notation in this paper. Given sequences of real numbers $a_n$ and $b_n$, we say $a_n=O(b_n)$ if there exists $C>0$ such that $|a_n|\le C|b_n|$ for every $n$. We say $a_n=o(b_n)$ if $\lim_{n\to\infty} a_n/b_n=0$. We say $a_n=\Theta(b_n)$ or $a_n\asymp b_n$ if $a_n,b_n>0$ and $a_n=O(b_n)$ and $b_n=O(a_n)$. Finally, we write $a_n=\omega(b_n)$ if $a_n,b_n>0$ and $b_n=o(a_n)$.



\section{Minors}\label{sec:minor}

\begin{lemma}[Lemma 2 of \cite{altschuler2017inclusion}] 
\label{full row rank} 
 For every $m\le n$,
$$\pr\big(\rank(A_m)=m\big)=\prod_{i=0}^{m-1}(1-q^{i-n}).$$
\end{lemma}

\proof It follows immediately by the fact that given the first $i\le m-1$ column vectors of $A_m$ being linearly independent, the probability that the $(i+1)$-th column vector falls in the span of the first $i$ column vectors is equal to $q^{i-n}$. \qed 




\begin{lemma}\label{minor corank}
    If $M$ is a matroid containing $N$ as a minor, then $\corank(M)\ge \corank(N)$. 
\end{lemma}
 \proof
    Let $e\in E(M)$.
    In both the case of deleting $e$ or contracting $e$, 
    the rank of $M$ decreases by at most one, and the size of the ground set of $M$ decreases by exactly one.
    Thus, $\corank(M)\ge\corank(M\setminus e)$ and $\corank(M)\ge\corank(M/e)$.  
    The result follows. \qed

\medskip

{\em Proof of Theorem~\ref{very small minor}.}
    By Lemma \ref{minor corank},  $\tau_{\Nminor}\ge\tau_{\crkc}$. To prove that a.a.s.\ $\tau_{\Nminor}\le\tau_{\crkc}$, consider the following equivalent way of generating the process $A_m$ for $m\le \tau_{\crkc}$.
    As each column vector $v_i$ is drawn, we only expose whether $v_i$ lies in the subspace generated by $v_1,\ldots, v_{i-1}$. If it is, we colour the column red; otherwise we colour it blue. Stop the process when there are exactly $c$ red columns. The following claim follows by Lemma~\ref{full row rank}.

   \begin{claim}\label{claim:t}
   A.a.s.\ the first $n/2$ columns are blue.
   \end{claim}

    
   Next, we expose all the column vectors corresponding to the blue columns. Let $A_B$ denote the submatrix of $A_m$ composed of all blue columns of $A_m$. Since the column vectors of $A_B$ are linearly independent, $A_B$ is row equivalent to the identity matrix, with possibly a few zero rows  underneath. In other words, there is an invertible matrix $P$ such that $PA_B=\left[\substack{I\\ \boldsymbol{0}}\right]$ where $\boldsymbol{0}$ are a set of all-0 vectors.

   Finally, we expose the column vectors corresponding to the red columns. Note that
each red column vector is a uniformly random vector in the span of the blue column vectors generated before it. 
\begin{claim} \label{claim:span}
Suppose that $v$ is a red column vector and there are $i$ blue columns before $v$. Then,
$Pv\sim \left[\substack{[U_q]^{i\times 1}\\ \boldsymbol{0}}\right]$. 
\end{claim}
In other words, $Pv$ has the same distribution as the vector obtained by appending $n-i$ 0's after a uniformly random vector in $\FF_q^i$.

Consider $M=M[A_m]$. Take the first $n/2$ blue columns of $A_m$ and partition them into $t:=\lceil n/2r\rceil$ groups $I_1,\ldots,I_t$ (by discarding the remaining columns if $n/2$ is not divisible by $r$), where $I_1$ denotes the first $r$ blue columns, $I_2$ denotes the next $r$ blue columns, etc. The a.a.s.\ existence of at least $n/2$ blue columns in $A_m$ is guaranteed by Claim~\ref{claim:t}. Let $X_j$ be the matrix obtained from $A_m$ by contracting all blue columns except for the blue columns in $I_j$ for $1\le j\le t$. Then, each $X_j$ has form $[I_{r\times r} \mid R_j]$.

\begin{claim}\label{claim:R}
$(R_j)_{j=1}^t$ are mutually independent, and each $R_j\sim [U_q]^{r\times c}$.
\end{claim}

Since $N$ is $\FF_q$ representable, we may represent $N$ by a rank-$r$ matrix of form $[I_{r\times r}\mid R]$ for some $r\times c$ matrix $R$ over $\FF_q$. By definition, $M$ contains $N$ as a minor if $R_j=R$ for some $1\le j\le t$. By Claim~\ref{claim:R}, this occurs with probability $q^{-rc}$ for each $1\le j\le t$. Moreover, all $R_j$ are independent. Thus, the probability that $M$ has an $N$-minor is at least
$$
1-(1-q^{-rc})^t \ge 1- \exp\left(-\frac{n}{2r}q^{-rc}\right) = 1-o(1),
$$
since $rq^{rc}=o(n)$.
\qed

It remains to prove Claims~\ref{claim:span} and~\ref{claim:R}.

\noindent{\em Proof of Claim~\ref{claim:span}.~ } Let $v_1,\ldots, v_i$ denote the $i$ blue column vectors that appear before $v$. Let $z_1,\ldots, z_i$ be i.i.d.\ uniform random variables in $\FF_q$. Since $v$ is a uniform random vector in $\spn{v_1,\ldots,v_i}$, the span of $v_1,\ldots,v_i$, $v\sim \sum_{j=1}^i z_j v_j$. Hence,
\[
Pv \sim  \sum_{j=1}^i z_j Pv_j.
\]
The claim follows by the distribution of $z_1,\ldots, z_i$ and the fact that $P[v_1,\ldots, v_i]=\left[\substack{I_{i\times i}\\ \boldsymbol{0}}\right]$. \qed

\ss

\noindent{\em Proof of Claim~\ref{claim:R}.~ } Let $A_R$ be the matrix formed by the red columns of $A_m$. By Claim~\ref{claim:t} we may assume that the first $n/2$ columns of $A_m$ are all blue, and thus by Claim~\ref{claim:span}, the submatrix of $PA_R$ formed by the first $n/2$ rows has distribution $[U_q]^{n/2\times c}$. The claim follows by noticing that $R_1$ is the first $r$ rows of $PA_R$, $R_2$ is the next $r$ rows of $PA_R$, etc. \qed 
\medskip

{\em Proof of Theorem~\ref{thm:point-prob}.}\ By Theorem~\ref{very small minor}, $\tau_{\Nminor}=n+k$ if $\corank(A_{n+k-1})=c-1$ and $\corank(A_{n+k})=c$, which happens if $\rank(A_{n+k-1})=\rank(A_{n+k})=n+k-c$. Our derivation of $\pr(\rank(A_{n+k-1})=\rank(A_{n+k})=n+k-c)$ is an easy adaptation of the proof of~\cite[Fact 3]{kordecki1991random}. For each $1\le j\le n+k-c$, let $u_j$ be 
the number of column vectors $v$ that are added in the process $(A_m)_{m=1}^{n+k}$ that lies in the subspace $\S$ generated by the column vectors added before $v$ when the dimension of $\S$ is equal to $j$. Then, letting $u:=\sum_{j=0}^{n+k-c} u_j$, $\rank(A_{n+k-1})=\rank(A_{n+k})=n+k-c$ if and only if $u=c$ and $u_{n+k-c}\ge 1$. Moreover, all $u_j$s are independent random variables, and for each $1\le j\le n+k-c-1$, $u_j$ has geometric distribution with probability $q^{j-n}$. It follows then that 
\begin{eqnarray*}
&&\pr\Big(\rank(A_{n+k-1})=\rank(A_{n+k})=n+k-c\Big)\\
&&\hspace{1cm}=[z^c] \left(\prod_{j=0}^{n+k-c-1}\sum_{h=0}^{\infty}(zq^{j-n})^h (1-q^{j-n}) \right)\left(\sum_{h=1}^{\infty}(zq^{k-c})^h\right)\\
&&\hspace{1cm}=[z^c] (1+O(q^{-n})) \beta_{k,c} zq^{k-c} \prod_{t=c-k}^{n} \frac{1}{1-zq^{-t}}\\
&&\hspace{1cm}=q^{k-c} \beta_{k,c} [z^{c-1}] (1+O(q^{-n}+zq^{-n})) \prod_{t=c-k}^{\infty} \frac{1}{1-zq^{-t}}\\
&&\hspace{1cm}=q^{k-c} \beta_{k,c} [z^{c-1}] (1+O(q^{-n}+zq^{-n})) \prod_{t=0}^{c-k-1} (1-zq^{-t}) \prod_{t=0}^{\infty} \frac{1}{1-zq^{-t}}.
\end{eqnarray*}
By Euler's formula (see~\cite[Corollary 2.2]{andrews1998theory}), if $|t|<1$ and $|z|<1$ then
\[
\prod_{i=0}^{\infty}\frac{1}{1-zt^i}=1+\sum_{i=1}^{\infty}\frac{z^i}{\prod_{j=1}^i(1-t^j)}.
\]
Thus,
\begin{eqnarray*}
&&\pr\Big(\tau_{\Nminor}=n+k\Big)\\
&&\hspace{1cm}=q^{k-c} \beta_{k,c} [z^{c-1}] (1+O(q^{-n}+zq^{-n})) \left(\prod_{t=0}^{c-k-1} (1-zq^{-t})\right) \left(1+\sum_{i=1}^{\infty}\frac{z^i}{\prod_{j=1}^i(1-q^{-j})}\right)\\
&&\hspace{1cm}\sim \beta_{k,c} q^{k-c}\left( [z^{c-1}] \prod_{t=0}^{c-k-1} (1-zq^{-t}) + \sum_{i=1}^{c-1} \frac{1}{\prod_{j=1}^i(1-q^{-j})}[z^{c-1-i}] \prod_{t=0}^{c-k-1} (1-zq^{-t})  \right).
\end{eqnarray*}
The theorem follows. \qed
\ss

\noindent {\em Proof of Corollary~\ref{thm:one_step}.~} The claim that $\pr(\tau_{\Nminor}\le n+c)=1-o(1)$ follows by Theorem~\ref{thm:point-prob}.  Moreover, $C_{c,c}$ in  Theorem~\ref{thm:point-prob} is equal to $\gamma_{q,c}=\prod_{j=c}^{\infty} (1-q^{-j})>0.$ Hence, $\lim_{n\to\infty} \pr(\tau_{\Nminor}\le n+c-1)=1-\gamma_{q,c}$. \qed\ss

\noindent {\em Proof of Corollary~\ref{cor:minor}.~} Since $|N|\le (2-\eps)\sqrt{\log_q n}$, $rc\le (1-\eps)\log_q^n$ where $r=\rank(N)$ and $c=\corank(N)$, and thus $rq^{rc}=o(n)$. By Theorem~\ref{very small minor}, a.a.s.\ $\tau_{\Nminor}=\tau_{\crkc}$. Since $c=\omega(1)$, a.a.s.\ $\tau_{\crkc}=n+\omega(1)$ by Lemma~\ref{full row rank}. It follows then that a.a.s.\ $\rank(A_{\crkc})=n$ and consequently $\tau_{\crkc}=n+\corank(N)$.\qed\ss

Before proving Theorem~\ref{thm:PGr}, we present a probabilistic tool of Poisson approximation of the balls-into-bins model.

\begin{lemma}\label{Poisson}
    Suppose $b$ balls are placed into $k$ bins, independently and uniformly at random.
    Let ${\mathcal E}$ be the event that every bin gets at least one ball. Set $\lambda=b/k$. 
    Then $$\pr(\E)\le 2(1-e^{-\lambda})^k \le 2e^{-ke^{-\lambda}} 
            \qquad\text{and}\qquad \pr(\overline{\E})\le 2ke^{-\lambda}.$$
            
\end{lemma}

\proof Let $Y_1,\ldots, Y_k$ be independent Poisson variables each with mean $\lambda$. Then, the distribution of the number of balls in bins is the same as $(Y_1,\ldots, Y_k)$ conditioned to $\sum_{i=1}^k Y_i=b$ (see e.g.\ \cite[Theorem 5.6]{mitzenmacher2017probability} for a proof). By Theorem 5.10 of~\cite{mitzenmacher2017probability} (with $f(x_1,\ldots, x_k)$ be the indicator variable that $x_i\ge 1$ for every $1\le i\le k$ or $f(x_1,\ldots, x_k)$ be the indicator variable that $x_i=0$ for some $1\le i\le k$), 
\[
\pr(\E)\le 2\pr(Y_i\ge 1 \forall i)=2(1-e^{-\lambda})^k;\quad \pr(\overline{\E})\le 2\pr(Y_i= 0\ \text{for some}\ i) \le 2 ke^{-\lambda}. \qed
\]

We also need the following lemma from~\cite{altschuler2017inclusion} concerning the distribution of a uniformly random vector in $\FF_q^n$ after a change of basis. 
\begin{lemma}\label{change of basis} \cite[Lemma 5]{altschuler2017inclusion}
    Suppose that $P\in\FF_q^{n\times n}$ is invertible, then $PA_m\sim[U_q]^{n\times m}$. 
\end{lemma}

Next, we assume that $r=r(n)\to\infty$. Recall that
\[
\zeta = \gbinom{r}{1}_q=\frac{q^r-1}{q-1}. 
\]
Note that $PG(r-1,q)$ 
has $\zeta$ elements.


\noindent {\em Proof of Theorem~\ref{thm:PGr}.~} For part (a), let $f=\omega(\zeta)$ and $f=o(\zeta\log\zeta)$. First we prove that a.a.s.\ 
$$\tau_{\PGr}\le n+\zeta\log\zeta +2f.$$
    Set $m=n+\zeta\log\zeta+2f$.
    By Lemma \ref{full row rank}, a.a.s.\ the first $n+f$ columns of $A_{m}$ have rank $n$. 
    Following these, there are $b=\zeta\log \zeta+f$ columns.
    After a change of basis, we obtain the following matrix that is row equivalent to $A_m$:
    \[
    [I_{n\times n} \quad * \quad B],
    \]
    where $I_{n\times n}$ is the $n$ by $n$ identity matrix, $*$ is a set of $f$ columns, and $B$ is obtained from the above $b$ columns after the change of basis. By
    Lemma \ref{change of basis}, $B\sim [U_q]^{n\times b}$. 
    
   Deleting the $f$ columns in $*$ and contracting all but the first $r$ columns in $I_{n\times n}$ we obtain
    \[
    [I_{r\times r}  \quad B_{r}],
    \]
    where $B_r$ is the $r\times b$ matrix obtained from the first $r$ rows of $B$. Hence, $B_{r} \sim  [U_q]^{r\times b}$. By definition $M[I_{r\times r}  \ \ B_{r}]$ is a minor of $A_m$. It is thus sufficient to prove that $B_r$ contains all elements in $PG(r-1,q)$. (It suffices to prove that $B_r$ contains all elements other than those already contained in $I_{r\times r}$. However it does not change the bound in any significant way.)

    Consider each element of $PG(r-1,q)$ as a bin and consider each column of $B_r$ as a ball. We say a ball $j$ is thrown into a bin $z$ if the $j$-th column vector of $B_r$ corresponds to one of the $q-1$ vectors associated to the $z$-th bin. Hence, $B_r$ contains all elements in $PG(r-1,q)$ if and only if every bin receives at least one ball. A ball here corresponds to a nonzero column vector, and it is easy tho show that a.a.s.\ at most $f/2$ of the $b$ columns can be zero columns. Hence, the total number of balls is at least $b-f/2$, and the total number of bins is equal to $\zeta$. Setting $\lambda=(b-f/2)/\zeta=\log\zeta +f/2\zeta$ and by Lemma~\ref{Poisson}, 
    \[
    \pr(\tau_{\PGr}>m)\le 2 \zeta e^{-\lambda}=O(q^r e^{-\log\zeta -f/2\zeta})=O(\exp(-f/2\zeta+O(1)))=o(1),
    \]
    as $\log \zeta = r\log q+O(1)$.
    
For part (b), the upper bound immediately follows from part (a). For the lower bound, fix $\varepsilon>0$ and we prove that if $r=\omega(\log n)$ then a.a.s.\ $\tau_{\PGr}\ge (1-\varepsilon)\zeta\log \zeta$. Set $m=(1-\varepsilon)\zeta\log \zeta$. By Lemma~\ref{full row rank}, we may assume that $A_m$ has rank $n$. For each $J\subseteq [m]$ where $|J|=n-r$, let $X_J$ be the indicator variable that $A_J$ has rank $n-r$, and the contraction of columns in $J$ produces a matroid that contains $PG(r-1,q)$ as a minor. Let $X=\sum_{J} X_J$ over all such subsets $J$. 
We claim that for every $J$, 
\begin{equation}
\ex X_J\le 2\exp(-\zeta^{\varepsilon}). \label{eq:XJ}
\end{equation}
Then,
$$
\ex X=2\binom{m}{n-r} \exp(-\zeta^{\varepsilon})\le \exp\left(n\log m -\zeta^{\varepsilon}\right)=\exp(-\zeta^{\varepsilon}+O(nr))=o(1),
$$
where the last equation above holds as $r=\omega(\log n)$. The lower bound for (b) follows by the Markov inequality. It  only remains to prove~\eqn{eq:XJ}. \ss

    \noindent {\em Proof of~\eqn{eq:XJ}.~} Similarly as before, the rank of $A_m$ is a.a.s.\ $n$ and the contraction of columns in $J $ where $|J|=n-r$ produces a matrix $[I_{r\times r} \ \ B]$, where each column of $B$
is a uniform random vector in $\FF_q^r$, provided that $A_J$ has rank $n-r$. Moreover, $B$ has $b=m-(n-r) \le (1-\varepsilon)\zeta\log\zeta$ columns. By Lemma~\ref{Poisson} (with $\lambda=b/\zeta\le (1-\varepsilon)\log\zeta$),
$$
\pr(X_J=1)\le 2 \exp\big(-\zeta e^{-(1-\varepsilon)\log \zeta }\big)=2\exp(-\zeta^\varepsilon).
$$    
    
Finally, for part (c), the upper bound is again implied by part (a), noticing that $\zeta\log \zeta=o(n)$ for the range of $r$ in part (c), and the fact that $PG(r-1,q)$ contains every $\FF_q$-representable minors of rank $r$. The lower bound follows since a.a.s.\ $A_m$ is a free matroid if $m-n\to-\infty$ and thus a.a.s.\ $\tau_{\Nminor}\ge n-\omega(1)$. \qed

\section{Circuits}
\label{sec:circuits}

\noindent {\em Proof of Corollary~\ref{cor:circuits}.~} For (a), set $m=cq^{n/k}$ where $c>0$ is fixed. Then by Theorem \ref{thm:old}(a), 
    $$\Pr(M[A_m]\ \text{has no $k$-circuits})\sim\exp\left(-\frac{(q-1)^{k-1}c^k}{k!}\right).$$
    Moreover, the above probability tends to 1 if $c\to0$,  and tends to 0 if $c\to\infty$.
    Therefore, $\tau_{\kCircuit}=\Theta_p(q^{n/k})$.

For (b), assume that $k,m\to\infty$ and $k=o(m)$. Let $\mu_k$ be as defined in Theorem~\ref{thm:old}. Then, 
\begin{eqnarray}
\log \mu_k &=&  \log\tbinom{m}{k} + k\log(q-1)-n\log q, \label{eq:log_mu}
\end{eqnarray}
where, by Stirling's formula,
    \begin{align}
        \log\tbinom{m}{k} &= m\log m-k\log k-(m-k)\log(m-k)+o(1) \nonumber\\
&=m\log m - k(\log m+\log k/m) - (m-k)(\log m + \log (1-k/m))+o(1)\nonumber\\
        &= k\log\big(\tfrac{m}{k}\big)+(m-k)\frac{k}{m}+O(k^2/m)=k\log\big(\tfrac{em}{k}\big) +o(k). \label{eq:log_binomial}
    \end{align}
Fix $c>0$. Setting $m=c\tfrac{1}{e(q-1)}kq^{n/k}$, we find that $k=o(m)$ and thus by~\eqn{eq:log_mu} and~\eqn{eq:log_binomial} we obtain
    $$\log \mu_k  = k\log(cq^{n/k})+o(k)-n\log q = k\log c+o(k).$$
It follows now that $\mu_k=o(1)$ if $c<1$ and $\mu_k=\omega(1)$ if $c>1$.
Thus, part (b) follows by Theorem \ref{thm:old}(b). 
    
    For part (c), we need the following claim about the function $g_a(y)$.
    
\begin{claim}\label{unique root}
    For every $0<a\le 1$, there exists a unique $b$ such that $g_a(b)=0$.
    Moreover, $g'_a(b)>0$. 
\end{claim}

    Let $b$ be the unique root of $g_a(y)$.
    Set $m=cbn$ for some fixed $c>0$.
    By Stirling's formula and a similar calculation as before, provided that $cb>a$,
    \begin{eqnarray*}
    \log \mu_k &=& - k\log (k/m) - (m-k) \log (1-k/m)  + k\log(q-1)-n\log q+ O(\log n)  \\
    &\sim& n\left(-a\log a +cb\log cb -(cb-a)\log(cb-a)+a\log(q-1)-\log q\right)=ng_a(cb).
    \end{eqnarray*}
    By the definition of $b$ and Claim~\ref{unique root}, $\mu_k=o(1)$ if $c<1$ and $\mu_k=\omega(1)$ if $c>1$. Part (c) follows. It only remains to prove Claim~\ref{unique root} and verify that $b>a$ (hence $cb>a$ for $c=1\pm\eps$ for every sufficiently small $\eps>0$). 

\noindent {\em Proof of $b>a$.~} This follows immediately from the facts that $g_a(y)$ is defined on $y\ge a$ and that $g_a(a)=a\log(q-1)-\log q < 0$. 

\noindent {\em Proof of Claim~\ref{unique root}.~} We find that $g'_a(y)=\log y-\log(y-a)>0$ for all $y>a$. Moreover, $\lim_{y\to\infty} g_a(y)=\infty$, and we have shown that $g_a(a) < 0$. It follows that $g_a(y)$ has a unique root $b>a$. \qed \medskip

\noindent {\em Proof of Corollary~\ref{cor2:circuits}.\ } Part (b) follows by Remark~\ref{remark:g} that $b(1)<2$ (the proof of Remark~\ref{remark:g} is given in the Appendix). For part (a), let $a^*=1-q^{-1}$ and let $b=b(a^*)=1$. By Remark~\ref{remark:g} and Corollary~\ref{cor:circuits}(c), a.a.s.\ at some step $m=(1+o(1))n$, $M[A_m]$ has a circuit whose length is asymptotic to $a^*n$. It remains to show that a.a.s.\ the first circuit cannot have length that is not asymptotic to $a^* n$. Fix $\eps>0$. We prove that there exists $c=c(\eps)>0$ such that  a.a.s.\ there is no circuit of length greater than $(a^*+\eps)n$ or shorter than $(a^*-\eps)n$ by step $m=(1+c)n$. Note that 
the expected number of circuits of length $k$ in $M[A_m]$ is asymptotic to $\mu_k$ given in Theorem~\ref{thm:old}(b). For every $k\ge (a^*+\eps)n$ or $k\le (a^*-\eps)n$, let $a_k=\lim_{n\to\infty} k/n$ (without loss of generality we may assume that $a_k$ exists by the subsubsequence principle) and let $b_k$ be the unique root of $g_{a_k}(y)$. By the condition on $k$ and since $b$ is strictly convex by Remark~\ref{remark:g}, it follows that $b_k>b(a^*)+\delta=1+\delta$ for some fixed $\delta=\delta(\eps)>0$. Let $m=(1+\delta/2)n= c' b_k n $ for some $c'\le 1-\delta/4$. We have shown in the previous proof that $\log \mu_{k}\sim ng_{a_k}(c'b_k)< -c''n$ for some fixed $c''=c''(\eps)>0$, as $c'<1$. Hence, the probability that $M[A_m]$ has a circuit of length greater than $(a^*+\eps)n$ or shorter than $(a^*-\eps)n$ is at most $ne^{-c''n}=o(1)$ by the union bound (over all $k$ such that $k\ge (a^*+\eps)n$ and $k\le (a^*-\eps)n$) and the Markov inequality. \qed

\section{Connectivity}
\label{sec:cconectivity}

\begin{lemma}[Proposition 3 of~\cite{cunningham}] \label{monotonicity}
    If $M=(E,\mc I)$ contains a vertically $k$-connected submatroid with the same rank as $M$, then $M$ is vertically $k$-connected.
\end{lemma}

\begin{proof}
    Let $e\in E$ such that $M\setminus e$ has the same rank as $M$.
    That is, $\rank(E\setminus\{e\})=\rank(E)$.
    It suffices to show that if $M$ is not vertically $k$-connected, then neither is $M\setminus e$.
    Therefore, suppose $M$ has a vertical $\ell$-separation $(X,Y)$ for some $\ell<k$.
    That is, $\rank(X)+\rank(Y)\le\rank(E)+\ell-1$ and $\rank(X),\rank(Y)\ge\ell$.
    Without loss of generality we may assume that $e\in X$.
    Let $X'=X\setminus\{e\}$.
    If $\rank(X')=\rank(X)$, then $(X',Y)$ is a vertical $\ell$-separation of $M\setminus e$.
    If $\rank(X')=\rank(X)-1$, then $(X',Y)$ is a vertical $(\ell-1)$-separation of $M\setminus e$.
    In either case, $M\setminus e$ is not vertically $k$-connected. \qed
\end{proof}
\ss

\noindent {\em Proof of Theorem~\ref{thm:monotone}(a).~} Let $f\to\infty$ be a slowly growing function of $n$. Obviously, a.a.s.\ $\kappa(M[A_m])=1$ for steps $m\le n-f$, as $M[A_m]$ is a free matroid. We leave it as an easy exercise that a.a.s.\ $\kappa(M[A_m])$ remains one for all $n-f\le m\le n+f$. Finally, we know that a.a.s.\ the rank of $A_m$ is $n$ for $m=n+f$. Combining all,  a.a.s.\ $\kappa(M[A_m])=1$ for all $m\le n+f$, and $\kappa(M[A_m])$ is non-decreasing for all $m\ge n+f$ by Lemma~\ref{monotonicity}. \qed\ss

\noindent {\em Proof of Theorem~\ref{sublinear connectivity}.~} We first prove the upper bound in part (a) by extending the proof of Kelly and Oxley~\cite[Theorem 4.5]{kelly1984random}.  Fix $\eps>0$.   Set $m=n+(1+2\varepsilon)k\log_q (n/k)$. Let $f=\omega(1)$ be a slowly growing function of $n$. Let $D$ be the first $n-f$ columns of $A_m$. Let $\E_{\ell}$ be the event that $M[A_m]$ is vertically $\ell$-separated, $\rank(A_m)=n$, and all columns of $D$ are independent. Since a.a.s.\ all columns of $D$ are linearly independent and $\rank(A_m)=n$, we have
\begin{equation}
\pr(M[A_m]\ \text{is not vertically $k$-connected})\le o(1)+\sum_{\ell\le k-1}\pr(\E_{\ell}). \label{eq:prob_connect}
\end{equation}
The probability of $\E_{\ell}$ was upper bounded by Kelly and Oxley in the following lemma.

\begin{lemma}\label{lemma 4.7} (\cite[Lemma 4.7]{kelly1984random})  
 $\pr\big(\E_{\ell}\big) \le \sum_{j=\ell}^{\lfloor\frac{1}{2}(n+\ell-1)\rfloor} b(\ell,j)$ where
    
                   $$  \binom{m-|D|}{n+\ell-1-|D|}\binom{n+\ell-1}{j}\left(\frac{q^j+q^{n+\ell-1-j}-q^{\ell-1}}{q^n}\right)^{m-(n+\ell-1)}.$$
\end{lemma}

We will prove the following claim.
\begin{claim}\label{claim:b}
$\pr(M[A_m]\ \text{is not vertically $k$-connected})\le O(nk\cdot b(k-1,k-1))+o(1).$
\end{claim}


\noindent {\em Proof of Claim~\ref{claim:b}.~}
By~\eqn{eq:prob_connect} and Lemma \ref{lemma 4.7}, 
 it suffices to show that $b(\ell,j)$ is maximized at $(\ell,j)=(k-1,k-1)$.
    Fix $1\le \ell\le k-1$ and $\ell\le j<\lfloor\frac{1}{2}(n+\ell-1)\rfloor$.
    Then $$\frac{b(\ell,j)}{b(\ell,j+1)}=\frac{j+1}{n+\ell-1-j}
        \left(\frac{q^j+q^{n+\ell-1-j}-q^{\ell-1}}{q^{j+1}+q^{n+\ell-2-j}-q^{\ell-1}}\right)^{m-(n+\ell-1)}.$$
    Let $A$ be the fraction inside the parentheses above, and let $B=(q^2+1)/2q$. We first prove that $A\ge B$. Note that $A\ge B$ if and only if
    \begin{align*}
& 2q^{j+1}+2q^{n+\ell-j}-2q^{\ell} \ge q^{j+3}+q^{n+\ell-j}-q^{\ell+1}+q^{j+1}+q^{n+\ell-2-j}-q^{\ell-1},
\end{align*}
which holds if and only if
\[
   (q^2-1)q^{n+\ell-2-j}+q^{\ell-1}(q-1)^2 \ge (q^2-1)q^{j+1}.
  \]
    Since $j<(n+\ell-1)/2$, we have $n+\ell-2-j \ge j+1$. This verifies that $A\ge B$.
    Therefore, $$\frac{b(\ell,j)}{b(\ell,j+1)} \ge \frac{A^{m-(n+\ell-1)}}{n} \ge \frac{B^{m-n-k}}{n}\ge1,$$
    where the last inequality holds by the definition of $B$ and by the setting of $m$.
    Thus, $b(\ell,j)$ is decreasing in $j$.
    Next, write $b(\ell,\ell)=XYZ$, where
    $$X=\binom{m-n}{\ell-1},\quad Y=\binom{n+\ell-1}{n-1},\quad 
        Z=\left(\frac{q^\ell+q^{n-1}-q^{\ell-1}}{q^n}\right)^{m-n-\ell+1}.$$
    Then $Y$ is increasing in $\ell$. Since $2\ell<2k\le m-n$, so is $X$.
    Also, $Z$ is increasing in $\ell$ since its base is less than one and increasing in $\ell$, and its exponent is decreasing in $\ell$.
    Therefore, $b(\ell,\ell)$ is increasing in $\ell$, so $b(\ell,j)\le b(\ell,\ell)\le b(k-1,k-1)$, as required.\qed
\ss

    Observe that
     \begin{align*}
        b(k-1,k-1) 
        &= \binom{m-n+f}{k-2+f}\binom{n+k-2}{k-1}\left(\frac{q^{k-1}+q^{n-1}-q^{k-2}}{q^n}\right)^{m-n-k+2} \\
        &\le \left(\frac{4k\log_q(n/k)}{k}\right)^{k+f}\left(\frac{4n}{k}\right)^k
            \left(\frac{1+q^{k-n}}{q}\right)^{(1+\varepsilon)k\log_q(n/k)},
    \end{align*}
    where to derive the last inequality above we used the fact that $m-n-k+2=(1+2\eps)k\log_q(n/k)-k+2\ge (1+\eps)k\log_q(n/k)$ since $k=o(n)$.
    Therefore, using that $n-k\to\infty$, we have
    \small
    \begin{align*}
        &\log\big[nk\cdot b(k-1,k-1)\big] \\
        &\le \log(nk)+(k+f)\log\log(n/k)+k\log(n/k) +O(k) +(1+\varepsilon)k\log_q(n/k) (-\log q +O(q^{k-n}))\\
        &=-\eps k \log(n/k) +O(\log n + k \log\log (n/k)))
    \end{align*}
    which goes to $-\infty$ since $k=o(n)$ and $k=\omega(1)$.
    Thus, $nk\cdot b(k-1,k-1)$ tends to zero. By Claim~\ref{claim:b}, a.a.s.\ $M[A_m]$ is vertically $k$-connected.
    \ss

    Next we prove the lower bounds in part (a) and (b) by the second moment method. A straight application of the second moment method to all possible separations would lead to failure due to heavy correlations. Instead we carefully craft the counting structures that imply the existence of a certain type of vertical $(k-1)$-separations.   
    For a pair $(I,S)$, 
    where $I$ is a subset of $k-1$ columns of $A_{m}$ and $S$ is an $(n-1)$-dimensional subspace of $\FF_q^n$,
    define $X_{I,S}$ to be the indicator variable that
    \begin{enumerate}[(i)]
    \item $I\subseteq S^c$, and all columns of $I$ are linearly independent, and 
    \item All column vectors in $A_m\setminus I$ are in $S$, and $\rank(A_m\setminus I)\ge k-1$.
     \end{enumerate} 
     Let $I^c$ denote $A_m\setminus I$, the set of vectors not in $I$.
    Note that (ii) above implies that $k-1\le \rank(I^c)\le n-1$.
    Thus, $X_{I,S}=1$ for some $(n-1)$-dimensional subspace $S$ immediately implies that $(I,I^c)$ is an $(k-1)$-separation.
    Therefore, it suffices to show that $X:=\sum X_{I,S}$ is a.a.s.\ positive, where the summation is taken over all $(k-1)$-subset of columns of $A_m$ and all $(n-1)$-dimensional subspaces of $\FF_q^n$. 
    
    For each given $(I,S)$, the events in (i) and (ii) are independent. Let $v_1,\ldots, v_{k-1}$ be the column vectors in $I$. For each $1\le i\le k-1$, the probability that $v_i\in \spn{v_1,\ldots,v_{i-1}} \cup S$ conditional on that $v_1,\ldots,v_{i-1}$ are linearly independent, and that none of them are in $S$ is
    \[
   \left\{
   \begin{array}{ll}
    q^{-n}(q^{i-1}+q^{n-1}-q^{i-2}) & \text{if}\ i\ge 2\\
    q^{-1} & \text{if}\ i=1.
    \end{array}
    \right.
    \] 
     Thus,  the probability of event (i) is (provided that $k/n<1-\eps$ for some $\eps>0$)
     \[
    (1- q^{-1})\prod_{i=2}^{k-1}  (1-   q^{-n}(q^{i-1}+q^{n-1}-q^{i-2}) )\sim (1-q^{-1})^{k-1}.
     \] 
    Similarly, the probability of events (ii) is asymptotic to $q^{-(m-k+1)}$. Consequently,
    $$\ex X_{I,S}\sim\mu:= (q-1)^{k-1}q^{-m} \quad\text{and}\quad
    \ex X\sim \binom{m}{k-1}\gbinom{n}{1}_q \mu.$$ 
    For part (a), set $m=n+(1-\varepsilon)k\log_q(n/k)$. For part (b), set $m=(1+\alpha)n$ where $\alpha$ satisfies~\eqn{eq:LB_connect}.
    We first verify that $\log\ex X=\omega(1)$ in both parts.
    Suppose that $k=o(n)$ and $k=\omega(1)$. Then, $m=n+(1-\varepsilon)k\log_q(n/k)$ and hence,
    \begin{align*}
        \log\ex X & = k\log (n/k) + n\log q+(k-1)\log(q-1)-m\log q +O(k+\log n) \\
        &=\eps k \log(n/k) +O(k+\log n)\to\infty,
    \end{align*}
    implying that $\ex X=\omega(1)$.
    Suppose that $k\sim tn$ for some $0<t<1$. Then,  by~\eqn{eq:LB_connect},
    \begin{align*}
        \log\ex X & = \left(t\log\frac{1+\alpha}{t}+(1+\alpha-t)\log\frac{1+\alpha}{1+\alpha-t}+t\log(q-1)-\alpha\log q+o(1)\right)n\to \infty.
    \end{align*}

    Next, we prove that $\ex X(X-1)\le (1+o(1)) (\ex X)^2$.  Consider a pair $(X_{I_i,S_i},X_{I_j,S_j})$.
    Let $h=|I_i\cap I_j|$.
    Note that if $X_{I_i,S_i}X_{I_j,S_j}=1$ then
        $I_i\subseteq S_i^c,\ I_j\subseteq S_j^c,\ I_i^c \subseteq S_i,\ I_j^c \subseteq S_j $. It follows that
        $
         I_i\cap I_j\subseteq (S_i\cup S_j)^c,\ I_i\setminus I_j\subseteq S_j\setminus S_i,\ 
            I_j\setminus I_i\subseteq S_i\setminus S_j$ and  $(I_i\cup I_j)^c\subseteq S_i\cap S_j.$
            
            We further consider two cases.
  If $S_i\ne S_j$, then $\dim(S_i\cap S_j)=n-2$, and so $|S_i\cap S_j|=q^{n-2}$. Thus,
    \begin{align*}
        \ex X_{I_i,S_i}X_{I_j,S_j}
        &\le \pr\big(I_i\cap I_j\subseteq (S_i\cup S_j)^c\big) \pr\big(I_i\setminus I_j\subseteq S_j\setminus S_i\big)
            \pr\big(I_j\setminus I_i\subseteq S_i\setminus S_j\big) \pr\big((I_i\cup I_j)^c\subseteq S_i\cap S_j \big) \\
        &= \left(\frac{|(S_i\cup S_j)^c|}{q^n}\right)^{|I_i\cap I_j|} \left(\frac{|S_j\setminus S_i|}{q^n}\right)^{|I_i\setminus I_j|}
            \left(\frac{|S_i\setminus S_j|}{q^n}\right)^{|I_j\setminus I_i|} \left(\frac{|S_i\cap S_j|}{q^n}\right)^{|(I_i\cup I_j)^c|} \\
        &= (1-2q^{-1}+q^{-2})^h (q^{-1}-q^{-2})^{k-1-h} (q^{-1}-q^{-2})^{k-1-h} (q^{-2})^{m-2(k-1-h)-h} \\
        &= \left(\frac{q-1}{q}\right)^{2h}\left(\frac{q-1}{q^2}\right)^{k-1-h}
            \left(\frac{q-1}{q^2}\right)^{k-1-h}\left(\frac{1}{q^2}\right)^{m-2(k-1-h)-h}
        = \mu^2.
    \end{align*}
    On the other hand, if $S_i=S_j$, then  $X_{I_i,S_i}X_{I_j,S_j}=1$ implies that
$I_i\cap I_j\subseteq S_i^c,\ I_i\setminus I_j\subseteq \emptyset,\ 
            I_j\setminus I_i\subseteq \emptyset,\ (I_i\cup I_j)^c\subseteq S_i$, which implies that 
        $I_i=I_j$ and $X_{I_i,S_i}=1$.
    Therefore, $\ex X_iX_j$ is nonzero only if $i=j$, in which case, $\ex X_{I_i,S_i}X_{I_j,S_j}=\mu$.
    Thus, \begin{align*}
        \ex X^2 &= \sum_{(I_i,S_i), (I_j,S_j)} X_{I_i,S_i}X_{I_j,S_j}
        = \sum_{I_i,I_j,S_i\neq S_j} \ex X_{I_i,S_i}X_{I_j,S_j}+\sum_{I_i=I_j,S_i=S_j} \ex X_{I_i,S_i}X_{I_j,S_j} \\
        &= \binom{m}{k-1}\gbinom{n}{1}\binom{m}{k-1}\left(\gbinom{n}{1}-1\right)\mu^2+\ex X 
        \sim (\ex X)^2.
    \end{align*}
    By Chebyshev's inequality, a.a.s.\ $X>0$, and therefore there exists a vertical $(k-1)$-separation. Now Theorem~\ref{sublinear connectivity} follows by the definition of vertical connectivity and Theorem~\ref{thm:monotone}(a). \qed\ss
    
    \noindent {\em Proof of Theorem~\ref{thm:monotone}(b,c).~} By Theorem~\ref{sublinear connectivity}, $\kappa(M[A_m])=o(n)$ if $m=n+o(n)$. On the other hand, $\text{gir}(M[A_m])=\Theta(n)$ for all $m=n+o(n)$, by Corollary~\ref{cor:circuits}. It is easy to see that a.a.s.\ for every step $m$ after the creation of the first circuit, $A_m$ is not isomorphic to any uniform matroid. Thus, part (b) follows now by Proposition~\ref{p:tutte}.

Part (c) follows by Proposition~\ref{p:tutte}, and the fact that a.a.s.\ $\kappa(M[A_m])$ is monotonely non-decreasing by part (a), and that $\text{gir}(M[A_m])$ is monotonely non-increasing.\qed
    

\section{Critical number}
\label{sec:colouring}

The following two lemmas follow from well known results in counting  subspaces of a vector space (see e.g.\ page 162 of~\cite{oxley2011matroid}). We include proofs for self-containment. Recall that we write $a_n\asymp b_n$ if $a_n=\Theta(b_n)$.
\begin{lemma} \label{gbinom approxoximation} 
    For any $0\le k\le n$, $$\gbinom{n}{k}_q = \gbinom{n}{n-k}_q \asymp q^{k(n-k)},$$
    which is the number of $k$-dimensional (and $(n-k)$-dimensional) subspaces of $\FF_q^n$.
\end{lemma}

\proof By definition,
\[
\gbinom{n}{k}_q=\gbinom{n}{n-k}_q=\prod_{i=0}^{k-1}\frac{q^{n-i}-1}{q^{k-i}-1}\asymp \prod_{i=0}^{k-1}\frac{q^{n-i}}{q^{k-i}} = q^{k(n-k)}.\qed 
\]

\vspace{0.3cm}

\begin{lemma} \label{subspace intersection} 
    Let $S$ be any $k$-dimensional subspace of $\FF_q^n$,
    and let $N(n,k,j,\ell)$ be the number of $j$-dimensional subspaces of $\FF_q^n$ that intersect with $S$ in dimension $\ell$. 
    Then $$N(n,k,j,\ell)=q^{(k-\ell)(j-\ell)}\gbinom{k}{\ell}_q\gbinom{n-k}{j-\ell}_q.$$
    \end{lemma}
\proof 
There are $\gbinom{k}{\ell}_q$ ways to choose an $\ell$-dimensional subspace $U$ of $S$.

Given $U$, there are $q^n-q^k$ vectors that are not in $S$. Adding any such vector into $U$ results an extension of $U$ into an $(\ell+1)$-dimensional subspace $U_1$ such that $S\cap U_1=U$. Then, there are $q^n-q^{k+1}$ vectors that are not in $S+ U_1$ whose addition extends $U_1$ to an $(\ell+2)$-dimensional subspace $U_2$. Repeat this, and we find that there are  
\[
\prod_{i=0}^{j-\ell-1}(q^n-q^{k+i})
\]
ways to extend $U$ to a $j$-dimensional subspace $W=U_{j-\ell}$ whose intersection with $S$ is $U$. However, by the same counting scheme (by considering vectors in $W\setminus U_i$, $0\le i\le j-\ell-1$ with $U_0=U$ this time), each such $j$-dimensional subspace $W$ can be constructed in $(q^j-q^{\ell})(q^j-q^{\ell+1})\cdots (q^j-q^{j-1})$ ways. It follows now that
$$N(n,k,j,\ell)=\gbinom{k}{\ell}_q\prod_{i=0}^{j-\ell-1}\frac{q^n-q^{k+i}}{q^j-q^{\ell+i}}
=q^{(k-\ell)(j-\ell)}\gbinom{k}{\ell}_q\gbinom{n-k}{j-\ell}_q.\qed$$
    By Lemma~\ref{subspace intersection},
\begin{equation}
N(n,n-k,n-k,n-2k+h)=q^{(k-h)^2}\gbinom{n-k}{n-2k+h}_q\gbinom{k}{k-h}_q=q^{(k-h)^2}\gbinom{n-k}{k-h}_q\gbinom{k}{h}_q.\label{eq:q_binom}
\end{equation}
    By Lemma \ref{gbinom approxoximation}, $$N(n,n-k,n-k,n-2k+h)\asymp q^{(k-h)^2}q^{(k-h)(n-2k+h)}q^{h(k-h)}= q^{(n-k+h)(k-h)}.$$

\begin{lemma} \label{gbinom assymptotic}
  Let $N, M, k$ be positive integers such that $N\ge M\ge k$. Then
    $$\gbinom{N}{k}_q= q^{(N-M)k}\gbinom{M}{k}_q \left(1+O(q^{k-M})\right).$$
\end{lemma}

\begin{proof}
    \begin{align*}
        \gbinom{N}{k}_q
   &     = \frac{(q^N-1)(q^{N-1}-1)\cdots(q^{N-k+1}-1)}{(q^k-1)(q^{k-1})\cdots(q-1)} 
        = \frac{q^Nq^{N-1}\cdots q^{N-k+1}}{(q^k-1)(q^{k-1})\cdots(q-1)} \left(1+O(q^{k-N})\right)\\
   &     = q^{(N-M)k} \frac{q^Mq^{M-1}\cdots q^{M-k+1}}{(q^k-1)(q^{k-1})\cdots(q-1)} \left(1+O(q^{k-N})\right)\\
     &   =q^{(N-M)k} \frac{(q^M-1)(q^{M-1}-1)\cdots(q^{M-k+1}-1)}{(q^k-1)(q^{k-1})\cdots(q-1)}  \left(1+O(q^{k-N}+q^{k-M})\right)\\
     &=q^{(N-M)k} \gbinom{M}{k}_q \left(1+O(q^{k-M})\right).\qed
    \end{align*}
\end{proof}

\noindent {\em Proof of Theorem~\ref{thm:colouring}.~}  Fix $0<\eps<1$.   Let $S_1,\dots,S_r$, where $r=\gbinom{n}{k}_q$, be the set of $(n-k)$-dimensional subspaces of $\FF_q^n$. Let $E$ be the set of column vectors in $A_m$.
    For every $1\le i\le r$, let $X_i$ be the indicator variable for $S_i\cap E=\emptyset$,
    and let $X=\sum_{i=1}^r X_i$. Hence $M[A_{m}]$ is $k$-colourable if and only if $X>0$.
    It is easy to see that
     $$\ex X_i=\mu:=(1-q^{-k})^m \quad\text{and}\quad \ex X= \gbinom{n}{k}_q\mu.$$
     By Lemma \ref{gbinom approxoximation},
    $$\\log\ex X = k(n-k)\log q+m\log(1-q^{-k}),$$
    which goes to $\infty$ if $m=-(1-\eps)k(n-k)\log q/\log(1-q^{-k})$,
    and goes to $-\infty$ if $m=-(1+\eps)k(n-k)\log q/\log(1-q^{-k})$.
    
    It suffices now to show that $\ex X(X-1)\sim (\ex X)^2$ if $m=-(1-\eps)k(n-k)\log q/\log(1-q^{-k})$.
   Consider a pair of distinct $(n-k)$-dimensional subspaces $(S_i,S_j)$. Let $h=\dim(S_i\cap S_j)-(n-2k)$. Thus, $\max\{0,2k-n\}\le h\le k-1$.
    Note that $X_iX_j=1$ if and only if $(S_i\cup S_j)\cap E=\emptyset,$
    and $$|S_i\cup S_j|=|S_i|+|S_j|-|S_i\cap S_j|=2q^{n-k}-q^{n-2k+h}.$$
    Therefore, $$\ex X_iX_j =\left(1-\frac{|S_i\cup S_j|}{q^n}\right)^m =\pi_h:=(1-2q^{-k}+q^{-2k+h})^{m}.$$
  Notice that $\pi_0=\mu^2$.    
    For $0\le h\le k$, let $N_h$ denote the number of $(n-k)$-dimensional subspaces whose intersection with $S_i$ is $n-2k+h$ (notice that this number is independent of $S_i$).
    Then $$\ex X(X-1) = \sum_{i,j} \ex X_iX_j = \sum_{i=1}^r\sum_{h=\max\{0,2k-n\}}^{k-1}\ \sum_{\substack{\dim(S_i\cap S_j)\\=n-2k+h}}\ex X_iX_j 
        = \gbinom{n}{k}_q \sum_{h=\max\{0,2k-n\}}^{k-1}N_h\pi_h.$$
        \begin{claim}\label{claim:sum}  Suppose that $(q,k)\neq (2,1)$. Let $\Lambda=2\log n$.
      If $k\le n/2-\Lambda$ then  $\sum_{h=0}^{k-1}N_h\pi_h \sim N_0\pi_0$. If $k\ge n/2+\Lambda$ then $\sum_{h=2k-n}^{k-1}N_{h}\pi_h \sim N_{2k-n}\pi_{2k-n}$. If $|k-n/2|\le \Lambda$ then $\sum_{h=\max\{0,2k-n\}}^{k-1}N_h\pi_h \sim \sum_{h=\max\{0,2k-n\}}^{2\Lambda} N_h\mu^2$.
              \end{claim}
              We first consider the case that $(q,k)\neq (2,1)$. 
        By~\eqn{eq:q_binom} and Lemma \ref{gbinom assymptotic}, 
        \begin{eqnarray*}
        N_0&=&N(n,n-k,n-k,n-2k)=q^{k^2}\gbinom{n-k}{k}_q\sim \gbinom{n}{k}_q \quad \mbox{if $n/2-k=\omega(1)$}\\
        N_{2k-n}&=& N(n,n-k,n-k,0)= q^{(n-k)^2} \gbinom{k}{n-k}_q \sim \gbinom{n}{n-k}_q \quad \mbox{if $k-n/2=\omega(1)$.} 
        \end{eqnarray*}
    Thus, if $n/2-k\ge \log n$ then
    $$\ex X^2 \sim 
    \gbinom{n}{k}_q N_0\pi_0 \sim \gbinom{n}{k}^2_q \mu^2 = (\ex X)^2.
    $$
    On the other hand, if $k-n/2\le\log n$ and $n-k\ge \log_q n+ \log_q\log n + \omega(1)$ then
    \begin{eqnarray*}
    \ex X(X-1) &\sim&  \gbinom{n}{k}_q N_{2k-n}\pi_{2k-n} \sim \gbinom{n}{k}^2_q \left(1-2q^{-k}+q^{-n}\right)^m\\
    &=&\gbinom{n}{k}^2_q \left((1-q^{-k})^2+O(q^{-n})\right)^m = (\ex X)^2(1+O(mq^{-n})) \sim (\ex X)^2. 
    \end{eqnarray*}
    We know that
    $$
    \gbinom{n}{k}_q =\sum_{h=\max\{0,2k-n\}}^k N_h.
    $$
    However,
    $$
    \frac{N_h}{N_{h-1}}\asymp q^{-2h+2k-n+1} \le 1/q \quad \mbox{for all $\max\{0,2k-n\}+1\le h\le k$}.
    $$
    It follows that 
    $$
    \gbinom{n}{k}_q \sim \sum_{h=\max\{0,2k-n\}}^{\log n} N_h.
    $$
   Thus, $\ex X(X-1) \sim \gbinom{n}{k}_q^2\mu^2=(\ex X)^2$. The theorem for the case that $(q,k)\neq (2,1)$ follows by combining all the three ranges of $k$. 
   
   In the case $(q,k)= (2,1)$ the critical number jumps from one to two when the first even circuit appears, which occurs in some step $n+O_p(1)\sim n$. Hence the theorem holds for the case $(q,k)= (2,1)$ as well.   \qed
  \ss
  
  The proof of Claim~\ref{claim:sum} uses the following inequality.
  \begin{lemma}\label{lem:inequality}
  $(1-q^{-k})^2>kq^{-k}$ for all positive integers $q\ge 2$ and $k\ge 1$ except that $(q,k)=(2,1)$. 
  \end{lemma}
    \proof Let $f_q(x)=(1-q^{-x})^2-xq^{-x}$ for $x\ge 1$. We find that $f_q'(x)=q^{-x}((2-2q^{-x}+x)\log q-1)>0$. Moreover, $f_q(1)=1-3q^{-1}+q^{-2}>0$ except that $q=2$, and $f_q(2)=1-4q^{-2}+q^{-4}>0$ for all $q\ge 2$. The first assertion follows.
    Similarly letting $g(x)=(1-2^{-k})^2\log 2-k 2^{-2k+1}$ the second assertion follows by $g'(x)=2^{-2k+1}(2k \log 2 - 1+2^k-(\log 2)^2)>0$ and $g(2)>0$.
    \qed \ss

        \noindent {\em Proof of Claim~\ref{claim:sum}.~}
  Note that
\begin{eqnarray*}  
  \pi_h\le ((1-q^{-k})^2+q^{-2k+h})^m&\le &\mu^2\exp\left(m \frac{q^{-2k+h}}{(1-q^{-k})^2}\right),\quad \mbox{for all $\max\{0,2k-n\}\le h\le k$}.
  \end{eqnarray*}
Note also that
and $\pi_h \sim \mu^2$ if $k-h\ge 4\log n$.

$$
\frac{N_h}{N_0}\asymp q^{-h(n-2k+h)}\quad \frac{N_h}{N_{2k-n}}\asymp q^{-h(n-2k+h)}.
$$
It suffices to show that for every $h$, 
\begin{equation}
N_h\pi_h/N_0\pi_0=N_h\pi_h/N_0\mu^2=o(1/k). \label{eq:ratio}
\end{equation}
We consider the following cases of $k$:

In the first case we consider $k=O(1)$. For every $1\le h\le k$

$$
\log\frac{N_h\pi_h}{N_0\pi_0} = -hn\log q +m \frac{q^{-2k+h}}{(1-q^{-k})^2} +O(1) \le -n\log q\left(h -k \frac{q^{-k}}{(1-q^{-k})^2} +o(1) \right),
$$
where the above inequality used $-\log(1-q^{-k})\ge q^{-k}$.
Since $h\ge 1$, the above is $-\Theta(n)$ by Lemma~\ref{lem:inequality}.  Thus,~\eqn{eq:ratio} holds when $k=O(1)$.

In the second case we consider $k=\omega(1)$ and $n/2-k\ge \log n$:
\begin{eqnarray}  
\log\frac{N_h\pi_h}{N_{0}\pi_{0}} &=& -h(n-2k+h)\log q + m \frac{q^{-2k+h}}{(1-q^{-k})^2} +O(1)\nonumber\\
&\le &-h(n-2k+h)\log q + (1+O(q^{-k})) k(n-k)q^{-k+h} \log q.\label{eq:log_ratio}
\end{eqnarray}
  It suffices to prove that 
  \begin{equation}
  h(n-2k+h)> (1+\alpha)k(n-k)q^{-k+h}\label{eq2:inequality}
  \end{equation}
  for some constant $\alpha>0$.
   Let $0<\eps<1/8$. We further discuss two cases:
  
  {\em (a). $h/k\le 1-\eps$.} If $n-2k=\Theta(n)$. Then the left hand side of~\eqn{eq2:inequality} is $\Theta(n)$, whereas the right hand side above is $o(n)$ since $k=\omega(1)$. Thus~\eqn{eq2:inequality} holds. If $n-2k=o(n)$. Then, $k=\Theta(n)$. Hence the right hand side above is $o(1)$ whereas the left hand side is $\Theta(\log n)$. Thus~\eqn{eq2:inequality} holds.
  
  {\em (b). $h/k\ge 1-\eps$.} Let $c=k-h$. Then, $1\le c\le \eps k$. The inequality~\eqn{eq2:inequality} is equivalent to
  \begin{equation}
  q^c(k(n-k)-c(n-c))>(1+\alpha)k(n-k).\label{eq3:inequality}
  \end{equation}
  Since $c\le \eps k$ and $k\le n/2$, $c(n-c)\le 2\eps k(n-k)$. Hence, the left hand side of~\eqn{eq3:inequality} is at least $(1-2\eps) q^c k(n-c)\ge (3/4)q^c k(n-k)$, which is greater than the right hand side as $q\ge 2$ and $c\ge 1$.

  In the third case we consider $k-n/2\ge \log n$. As $k=\omega(1)$ we have that~\eqn{eq:log_ratio} still holds after replacing $N_0\pi_0$ by $N_{2k-n}\pi_{2k-n}$,  and thus it suffices to prove~\eqn{eq2:inequality}, i.e.\
  \begin{equation}
  q^{k-h} h (n-2k+h) > (1+\alpha) k(n-k) \quad \mbox{for all $2k-n+1\le h\le k-1$}. \label{eq4:inequality}
  \end{equation}
  Similarly as before, if $h/k\le 1-\eps$ then we can easily verify~\eqn{eq4:inequality}. Suppose that $h\ge (1-\eps)k$. Consider the derivative of $f(h)=q^{-h} x (n-2k+h)$ we find that $f'(h)=q^{-h}((\log q) h ( 2k -n  - h) + 2h - 2k + n)$. Taking the derivative of $(\log q) h ( 2k -n  - h) + 2h - 2k + n$ we obtain $( 2k- n - 2h)\log q + 2<0$ for all $(1-\eps)k\le h\le k-1$.  Hence, $f$ is concave in the interval $(1-\eps)k\le h\le k-1$. Thus, to verify~\eqn{eq4:inequality} it suffices to show that $f(h)>k(n-k)$ for $h=(1-\eps)k$ and for $h=k-1$. We have already argued that~\eqn{eq4:inequality} holds for $h=(1-\eps)k$. For $h=k-1$, the left hand side of~\eqn{eq4:inequality} is $q(k-1)(n-k-1)$ which is clearly greater than the right hand side with sufficiently small $\alpha$ as $q\ge 2$, $k\ge n/2$ and $n-k=\omega(1)$.
  
 In the final case let's consider the case that $k=n/2+O(\Lambda)$. It is easy to see that for $k$ in this range and for $h$ between $\max\{0,2k-n\}+1$ and $2\Lambda$, $\pi_h\sim \mu^2$. It is thus sufficient to prove that 
  $$
  \frac{N_h\pi_h}{N_{*}\pi_{*}}=o(1/n), \quad\mbox{where $*=\max\{0,2k-n\}$.}
  $$
  For simplicity we assume that $k\le n/2$ and the case that $k\ge n/2$ is symmetric.
  
  We have~\eqn{eq:log_ratio} and want to prove~\eqn{eq2:inequality} for every $1\le h\le 2\Lambda$. The left hand side is at least 1, and the right hand side is $o(1)$. So the inequality holds. \qed


\begin{thebibliography}{10}

\bibitem{altschuler2017inclusion}
J.~Altschuler and E.~Yang.
\newblock Inclusion of forbidden minors in random representable matroids.
\newblock {\em Discrete Mathematics}, 340(7):1553--1563, 2017.

\bibitem{andrews1998theory}
G.~E. Andrews.
\newblock {\em The theory of partitions}.
\newblock Number~2. Cambridge university press, 1998.

\bibitem{bansal2015number}
N.~Bansal, R.~A. Pendavingh, and J.~G. van~der Pol.
\newblock On the number of matroids.
\newblock {\em Combinatorica}, 35:253--277, 2015.

\bibitem{cooper2019minors}
C.~Cooper, A.~Frieze, and W.~Pegden.
\newblock Minors of a random binary matroid.
\newblock {\em Random Structures \& Algorithms}, 55(4):865--880, 2019.

\bibitem{cunningham}
W.~H. Cunningham.
\newblock On matroid connectivity.
\newblock {\em Journal of Combinatorial Theory Series B}, 30(1):94--99, 1981.

\bibitem{erdds1959random}
P.~Erd\H{o}s and A.~R\'{e}nyi.
\newblock On random graphs i.
\newblock {\em Publ. math. debrecen}, 6(290-297):18, 1959.

\bibitem{gao2023minors}
P.~Gao and P.~Nelson.
\newblock Minors of matroids represented by sparse random matrices over finite
  fields.
\newblock {\em arXiv preprint arXiv:2307.15685}, 2023.

\bibitem{kelly1984random}
D.~G. Kelly and J.~matroid.
\newblock On random representable matroids.
\newblock {\em Studies in Applied Mathematics}, 71(3):181--205, 1984.

\bibitem{kelly1982asymptotic}
D.~G. Kelly and J.~G. matroid.
\newblock Asymptotic properties of random subsets of projective spaces.
\newblock In {\em Mathematical Proceedings of the Cambridge Philosophical
  Society}, volume~91, pages 119--130. Cambridge University Press, 1982.

\bibitem{kelly1982threshold}
D.~G. Kelly and J.~G. matroid.
\newblock Threshold functions for some properties of random subsets of
  projective spaces.
\newblock {\em The Quarterly Journal of Mathematics}, 33(4):463--469, 1982.

\bibitem{knuth1974asymptotic}
D.~E. Knuth.
\newblock The asymptotic number of geometries.
\newblock {\em Journal of Combinatorial Theory, Series A}, 16(3):398--400,
  1974.

\bibitem{kordecki1988strictly}
W.~Kordecki.
\newblock Strictly balanced submatroids in random subsets of projective
  geometries.
\newblock In {\em Colloquium Mathematicum}, volume~2, pages 371--375, 1988.

\bibitem{kordecki1996small}
W.~Kordecki.
\newblock Small submatroids in random matroids.
\newblock {\em Combinatorics, Probability and Computing}, 5(3):257--266, 1996.

\bibitem{kordecki1991random}
W.~Kordecki and T.~{\L}uczak.
\newblock On random subsets of projective spaces.
\newblock In {\em Colloquium Mathematicae}, volume~62, pages 353--356, 1991.

\bibitem{kordecki1999connectivity}
W.~Kordecki and T.~{\L}uczak.
\newblock On the connectivity of random subsets of projective spaces.
\newblock {\em Discrete mathematics}, 196(1-3):207--217, 1999.

\bibitem{lowrance2013properties}
L.~Lowrance, J.~matroid, C.~Semple, and D.~Welsh.
\newblock On properties of almost all matroids.
\newblock {\em Advances in Applied Mathematics}, 50(1):115--124, 2013.

\bibitem{mayhew2011asymptotic}
D.~Mayhew, M.~Newman, D.~Welsh, and G.~Whittle.
\newblock On the asymptotic proportion of connected matroids.
\newblock {\em European Journal of Combinatorics}, 32(6):882--890, 2011.

\bibitem{mitzenmacher2017probability}
M.~Mitzenmacher and E.~Upfal.
\newblock {\em Probability and computing: Randomization and probabilistic
  techniques in algorithms and data analysis}.
\newblock Cambridge university press, 2017.

\bibitem{nelson2016almost}
P.~Nelson.
\newblock Almost all matroids are non-representable.
\newblock {\em arXiv preprint arXiv:1605.04288}, 2016.

\bibitem{oxley1984threshold}
J.~G. Oxley.
\newblock Threshold distribution functions for some random representable
  matroids.
\newblock In {\em Mathematical Proceedings of the Cambridge Philosophical
  Society}, volume~95, pages 335--347. Cambridge University Press, 1984.

\bibitem{oxley2011matroid}
J.~G. Oxley.
\newblock {\em Matroid theory, Second Edition}.
\newblock Oxford University Press, USA, 2011.

\bibitem{pendavingh2018number}
R.~Pendavingh and J.~Van Der~Pol.
\newblock On the number of bases of almost all matroids.
\newblock {\em Combinatorica}, 38:955--985, 2018.

\end{thebibliography}

\section*{Appendix}

\noindent{\em Proof of Remark~\ref{remark:g}.~} 
    Since $g_a(b(a))=0$, we have 
    \begin{equation}
    b(a)\log b(a)+a\log(q-1)=a\log a+(b(a)-a)\log(b(a)-a)+\log q\quad \text{for every $0<a\le 1$}. \label{eq2:g}
    \end{equation}
    Differentiating with respect to $a$ gives 
    \begin{equation}\label{eq:b_derivative}
    b'\big(\log b-\log(b-a)\big)=\log a-\log(q-1)-\log(b-a).
    \end{equation}
    Implicitly differentiating again gives 
    $$b''\big(\log b-\log(b-a)\big)ab(b-a) = a^2(b')^2-2abb'+b^2 = (ab'-b)^2 \ge0.$$
    Thus, $b''(a)\ge 0$ for every $0<a\le 1$, and so $b$ is convex.
    Moreover, we found that $(a,b)=(a^*,1)$ where $a^*=1-1/q$ satisfies~\eqn{eq2:g}, implying that $b(a^*)=1$. Moreover, plugging $(a,b)=(a^*,1)$ into~\eqn{eq:b_derivative} gives $b'(a^*)=0$, which means that $a^*$ minimizes $b(a)$.
    Lastly, observe that $g_1(2)>0$ and $g_a(1/a)\to-\log q$ as $a\to0$. Since for every fixed $a>0$, $g_a(y)$ is an increasing function of $y$ on $y\ge a$,
    it follows that $b(1)<2$ and that for $a$ sufficiently small, $b(a)>1/a$. Hence $b(a)\to\infty$ as $a\to0$. \qed

\end{document}